\documentclass[a4paper]{amsart}
\usepackage{amscd, amssymb, pb-diagram}
\usepackage{mathtools}
\usepackage{mathrsfs}
\usepackage{enumerate}
\usepackage[shortlabels]{enumitem}
\usepackage{amsmath}
\usepackage[margin=3.6cm]{geometry}
\usepackage{amscd, amssymb, pb-diagram}
\usepackage{tikz-cd}
\usepackage{amsthm} 
\usepackage[nospace,noadjust]{cite}
\usepackage{lmodern}
\usepackage{graphicx}
\graphicspath{{images/}{../images/}}
\usepackage[toc,page]{appendix}
\usepackage[normalem]{ulem}
\usepackage{emptypage}
\usepackage{hyperref}
\usepackage{subfiles}
\usepackage{datetime}

\usepackage{tensor}
\usepackage{wrapfig}
\usepackage{tikz}
\usepackage{multicol}
\usepackage[autostyle]{csquotes}

\allowdisplaybreaks

\def\ker{\operatorname{Ker}}

\def\sup{\operatorname{sup}}
\def\max{\operatorname{max}}
\def\min{\operatorname{min}}

\def\d{\operatorname{d}}
\def\LC{\operatorname{C_{c}^\infty}(G_A)}
\def\Aut{\operatorname{Aut}}

\makeatletter
\def\Ddots{\mathinner{\mkern1mu\raise\p@
\vbox{\kern7\p@\hbox{.}}\mkern2mu
\raise4\p@\hbox{.}\mkern2mu\raise7\p@\hbox{.}\mkern1mu}}
\makeatother

\numberwithin{equation}{section}

\tikzstyle{vertex}=[circle]
\tikzstyle{goto}=[->,shorten >=1pt,>=stealth,semithick]

\newtheorem{thm}{Theorem}[section]
\newtheorem{cor}[thm]{Corollary}

\newtheorem{lemma}[thm]{Lemma}
\newtheorem{prop}[thm]{Proposition}
\newtheorem{thm1}{Theorem}

\theoremstyle{definition}
\newtheorem{definition}[thm]{Definition}

\theoremstyle{remark}
\newtheorem{remark}[thm]{Remark}
\newtheorem{example}[thm]{Example}

\makeatletter
\@namedef{subjclassname@2020}{%
  \textup{2020} Mathematics Subject Classification}
\makeatother

\begin{document}

\title[Heat and isometries on CK--algebras]{Heat operators and isometry groups of Cuntz--Krieger algebras}
\author{Dimitris M. Gerontogiannis, Magnus Goffeng, and Bram Mesland}

\address{\scriptsize Dimitris Michail Gerontogiannis, IM PAN,  Jana i Jedrzeja \'Sniadeckich 8, 00-656 Warszawa, Poland}
\email{dgerontogiannis@impan.pl}

\address{\scriptsize Magnus Goffeng, Centre for Mathematical Sciences, Lund University, Box 118, 221 00 LUND, Sweden}
\email{magnus.goffeng@math.lth.se}

\address{\scriptsize Bram Mesland, Leiden University, Gorlaeus Gebouw, BW-vleugel, Einsteinweg 55, 2333 CC Leiden, The Netherlands}
\email{b.mesland@math.leidenuniv.nl}


\maketitle

\begin{abstract}
This paper introduces heat semigroups of topological Markov chains and Cuntz--Krieger algebras by means of spectral noncommutative geometry. Using recent advances on the logarithmic Dirichlet Laplacian on Ahlfors regular metric-measure spaces, we construct spectral triples on Cuntz--Krieger algebras from singular integral operators. These spectral triples exhaust odd $K$-homology and for Cuntz algebras we can compute their heat operators explicitly as Riesz potential operators. We also describe their isometry group in terms of the automorphism group of the underlying directed graph and prove that the Voiculescu noncommutative topological entropy vanishes on isometries. 
\end{abstract}

\section{Introduction}
The geometry and topology of manifolds can be encoded spectrally by means of suitable (pseudo)differential operators, such as the Laplace, Hodge--deRham or Dirac operator. This functional analytic approach leads to Connes’ paradigm of noncommutative geometry \cite{connesredbook}, which extends spectral geometry to the realm of noncommutative $C^{*}$-algebras. Here, the central notion is that of a spectral triple $(\mathcal{A},H,D)$. This consists of a pre-$C^{*}$-algebra $\mathcal{A}$, a Hilbert space $H$ and a self-adjoint operator $D$ with compact resolvent satisfying some compatibility axioms. The prototypical commutative example is that of a Dirac-type operator on a compact manifold. However, there are only few examples of large classes of noncommutative spectral triples with non-trivial topological as well as spectral features.

The class of Cuntz--Krieger algebras \cite{CK}, a family containing highly noncommutative $C^{*}$-algebras arising from topological Markov chains, is such a notable example. Early on, Cuntz--Krieger algebras were shown to satisfy $KK$-duality \cite{KPDual}, a topological feature that in the commutative setting is enjoyed by finite complexes. The noncommutative geometry of Cuntz--Krieger algebras was explored thereafter in \cite{GMCK, GMR ,GMON} where the authors obtained spectral triples that exhaust odd $K$-homology. Building on this, a study of the isometry group of Cuntz algebras was pursued in \cite{contirossi}. Moreover, related constructions have been carried out in the commutative case of the underlying shift space \cite{antoineputnam,farsietal} with a focus on metric properties and wavelet bases.

In this paper we revisit the noncommutative geometry of Cuntz--Krieger algebras in the spirit of semiclassical methods and spectral geometry. The key tools come from the recent developments made by the first and third listed author, in the harmonic analysis of Ahlfors regular metric-measure spaces \cite{GerMes}, and ideas \cite{GMON} put forward by the second and third listed author for Cuntz algebras. Notably, our present more geometric approach produces novel results already for the Cuntz algebras.

\subsection{Background}
The primordial idea of studying dynamical systems through the lens of operator algebras goes back almost a century to von Neumann, with operator algebraic methods imbuing his proof of the ergodic theorem. Nowadays, this line of thought is well-established throughout mathematics \cite{bostconnes,CK,GPS,lacanesh,voicent}. Noncommutative geometry extends this philosophy to a natural framework in which both rigid topological and fine analytic properties of dynamical systems can be studied on equal footing. 

The last decade has seen substantial progress regarding the noncommutative topology of non-isometric dynamical systems, most notably the hyperbolic dynamical systems encoded by the Ruelle algebras \cite{ruelle}. These are $C^*$-algebras satisfying KK-duality \cite{gerowhitzach,kamputwhit}. The Cuntz--Krieger algebras \cite{CK} arise as the zero-dimensional case in the class of Ruelle algebras.

More precisely, the Cuntz--Krieger algebra $O_{A}$ associated to a topological Markov chain with adjacency matrix $A$ encapsulates key aspects of the (singular) orbit space of the associated subshift of finite type. The noncommutative topology of Cuntz--Krieger algebras is captured through analytic $K$-homology and recovers the Bowen--Franks invariants \cite{CK}. Accessing this information computationally requires the construction of a spectral geometry compatible with non-isometric dynamics of the underlying subshift of finite type. 

Questions of dynamical systems and equivariance permeate noncommutative geometry \cite{connesredbook} and its applications. Nonetheless, a complete  reconciliation with spectral theory and dynamical systems is yet to be achieved. In fact, this is a hard problem that so far is satisfactorily understood only for isometric actions. 

The Baum--Connes conjecture \cite{bch} implies that up to homotopy all actions are isometric and proper. In light of the Baum--Connes conjecture’s veracity for large classes of groups \cite{higkasp,laffa}, it poses no restriction to restrict to isometric, proper actions for the purposes of index theory and $K$-theory.
However, for the purpose of explicit computation, the problem of obtaining a spectral geometry in the non-isometric case persists. 
The objective here is a systematic construction of spectral triples that gives access to their index theory via a $K$-homology class, their symmetries in the form of an isometry group \cite{park}, their heat kernel and heat trace asymptotics, and the associated Connes distances, also known as quantum or Monge--Kantorovi\v{c} metrics. 

The noncommutative geometry of Cuntz--Krieger algebras was first studied in \cite{GMCK}, motivated by the study of dimensional properties in noncommutative geometry. The study of spectral noncommutative geometry on these algebras was later pursued in \cite{GMR,GMON,DGMW,GRU}. However, those approaches did not provide a complete picture. The purpose of this paper is to refine the above mentioned works in a systematic way that has the potential to extend to dynamical systems of positive dimension.

\subsection{Main results}
The main results of the paper concern the noncommutative geometry of Cuntz--Krieger algebras $O_A$, which arises from a distinguished metric-measure space structure on the Deaconu--Renault groupoid $G_A$ with unit space the associated topological Markov chain $\Sigma_A$. The reader can find a review of topological Markov chains in Subsection \ref{subsec:topmark} and of the associated dynamics as encoded in the Deaconu--Renault groupoid $G_A$ in Subsection \ref{subsec:deac}. We decompose the groupoid $G_A$ as a disjoint union of bisections
\begin{equation}
\label{eq:bisec}
G_A=\bigsqcup_{\gamma\in I_{A}}G_\gamma,
\end{equation}
such that the metric-measure space structure on $\Sigma_{A}$ lifts to each bisection $G_\gamma$. We so arrive at an (extended) metric $\mathrm{d}_{G_A}$ and a measure $\mu_{G_A}$ fitting into an Ahlfors regular metric-measure space $(G_A,\mathrm{d}_{G_A},\mu_{G_A})$. The measure $\mu_{G_A}$ is such that $L^2(G_A,\mu_{G_A})$ canonically identifies with the GNS-representation of $O_A$ associated to the KMS-state of the gauge action.  

As an Ahlfors regular metric-measure space, $(G_A,\mathrm{d}_{G_A},\mu_{G_A})$ admits the logarithmic Dirichlet Laplacian \cite{GerMes}. The resulting operator $\Delta$, densely defined in $L^2(G_A,\mu_{G_A})$, takes the form
\begin{equation*}
\Delta f(g_1)=\int_{G_A} \frac{f(g_1)-f(g_2)}{\mathrm{d}_{G_A}(g_1,g_2)^\delta} \d \mu_{G_A}(g_2),
\end{equation*}
with $\delta$ being the fractal dimension of the underlying shift space. In addition, the decomposition \eqref{eq:bisec} of $G_{A}$ allows for the construction of a proper, continuous function 
\begin{equation}
\ell:G_{A}\to\mathbb{N}, \quad \ell|_{G_{\gamma}}:=|\gamma|,
\end{equation}
with $|\gamma|$ a naturally defined length on the index set $I_{A}$. The function $\ell$ gives rise to a suitable potential $V$. We so produce an odd spectral triple $(\LC,L^2(G_A,\mu_{G_A}),-\Delta+V)$ over $O_A$, with $\LC$ being the convolution $*$-algebra of compactly supported locally constant functions. The detailed construction of $D=-\Delta+V$ is found in Subsection \ref{subsec:spectrip}. 

The $K$-homology class of $(\LC,L^2(G_A,\mu_{G_A}),D)$ is in general non-trivial. Further, in the spirit of \cite{GMCK}, by varying appropriate parameters we obtain a family of spectral triples that exhausts the $K$-homology group $K^1(O_A)$, see Proposition \ref{lknlknlknad2}. A key novelty in the present paper is the fine spectral analysis of the operator $D$, provided by the following theorem.

\begin{thm1}
\label{firstmain}
The Hilbert space $L^2(G_A,\mu_{G_A})$ admits an orthonormal basis $$(\mathrm{e}_{\gamma},\mathrm{e}_{(\gamma,\nu,j)})_{(\gamma,\nu,j)\in \mathfrak{I}_A}\subset \LC\subset L^2(G_A,\mu_{G_A}),$$ of eigenfunctions for the operator $D$. The eigenvalues are given by 
\[D\mathrm{e}_\gamma=\mathrm{sgn}(\gamma) |\gamma|\mathrm{e}_\gamma,\quad D \mathrm{e}_{(\gamma,\nu,j)}=-(|\gamma|+\lambda_{s(\gamma)}^A(\nu))\mathrm{e}_{(\gamma,\nu,j)},\]
where the map $\mathrm{sgn}:I_{A}\to \{\pm 1\}$ is determined by Equation \eqref{eq:D}. 
The numbers $\lambda_{s(\gamma)}^A(\nu)$ are eigenvalues of the logarithmic Dirichlet Laplacian on $G_A$ and are explicitly given by 
$$\lambda_{s(\gamma)}^A(\nu)=1+\lambda_{\max}v_{\nu_1}\sum_{k=0}^{|\nu|-|s(\gamma)|-1}u_{\nu_{|\nu|-k-1}}(1-P_{\nu_{|\nu|-k-1}, \nu_{|\nu|-k}}),
$$ 
with $u$ and $v$ the Perron--Frobenius eigenvectors for $A$ and its transpose $A^T$, respectively, and $P_{i,j}$ is the transition probability of going from \textit{vertex} $i$ to \textit{vertex} $j$ in the finite directed graph associated to $A$. 
\end{thm1}

The reader can find the main computation underlying Theorem \ref{firstmain} in Subsection \ref{specsubsec}. It is of spectral interest to know that for fixed $\gamma$ the spectrum of $D$ restricted to the span of $(\mathrm{e}_{\gamma},\mathrm{e}_{(\gamma,\nu,j)})_{(\nu,j)}$ grows exactly logarithmically, see Proposition \ref{asjnkjnada}. In particular, the operator $D$ exhibits spectral similarity to the signed logarithm of an operator with Weyl like spectral behaviour, such as a Dirac operator on a finite-dimensional geometry. For the Cuntz algebras $O_{N}$, we prove that this is indeed the case. Here the topological Markov chain is the full $N$-shift, and the homogeneity of the example allows for explicit computations, see Theorem \ref{heatthm}.

\begin{thm1}
For the Cuntz algebra $O_N$ the heat operator $\mathrm{e}^{-t|D|}$ of $D$ is defined for every $t>\log(N)(1-N^{-1})^{-1}$ by the kernel $$K_t\in L^1(G_A\times G_A,\mu_{G_A}\times \mu_{G_A})\cap C(G_A\times G_A),$$ given by 
$$K_t(g_1,g_2)=\begin{cases}e^{-t|\gamma|}\left(h_\gamma(t)+H_\gamma(t)\mathrm{d}_\gamma(g_1,g_2)^{-\delta+\delta't}\right), & \text{if  }\; g_1,g_2\in G_{\gamma}\\
0, & \text{otherwise}.\end{cases}$$ 
Here $h_\gamma$ and $H_\gamma$ are explicit rational functions in $N$, $\mathrm{e}^{-tN^{-1}}$ and $\mathrm{e}^{-t}$ given in Proposition \ref{lnlknkojnad}. The distance $d_\gamma$ is the distance on the topological Markov chain pulled back to the bisection $G_\gamma$ via the source map. In particular, the heat operators $\mathrm{e}^{-t|D|}$ are Riesz potential operators. 
\end{thm1}

In Section \ref{secljnjasn}, we compute the isometry group of $(\LC,L^2(G_A,\mu_{G_A}),D)$ and study the noncommutative topological entropy as defined by Voiculescu. The reader can note that classically, the topological entropy of isometries vanishes. The results of Section \ref{secljnjasn} can be summarized in the following theorem.

\begin{thm1}
The isometry group $\mathsf{G}_D$ of the spectral triple $(\LC, L^2(G_A,\mu_{G_A}), D)$ is a compact Lie group isomorphic to $\mathbb T^N\rtimes \Aut(A)$, where $\Aut(A)$ is the automorphism group of the directed graph associated to $A$. Moreover, the Voiculescu noncommutative topological entropy vanishes on automorphisms of $O_A$ belonging to $\mathsf{G}_D$.
\end{thm1}

The reader should note that our computation of the isometry group does not contradict the main result in \cite{contirossi} as the two results concern different spectral triples. Further, we note that the situation on $O_A$ mimics the classical case of metric spaces where it also holds that the entropy of isometries vanishes. The first part of the preceding theorem can be found in the main text as Theorem \ref{thm:Isom_group} and the second part as Theorem \ref{vanishihsidnadin}. Moreover, the preceding theorem also says something about the strength of the isometry group of $(\LC, L^2(G_A,\mu_{G_A}), D)$ as an invariant: the isometry group recovers the number $N$ of Cuntz--Krieger generators of $O_A$ and the automorphism group of the associated directed graph. However, since many graphs have trivial automorphism groups, this is far from a complete invariant.\\

The techniques in this paper extend those put forward in \cite{GMON} and make essential use of the logarithmic Dirichlet Laplacian developed in \cite{GerMes}. In the course of our work, we noticed a subtle but crucial issue in the definition of the metric-measure space structure from \cite{GMON} on the groupoid $G_{A}$. We address the issue in Remark \ref{rem:Unbounded_com}, and the present paper retrieves the results announced in \cite{GMON} for the Cuntz algebras $O_N$.

\subsection*{Acknowledgements} The authors would like to thank Amaury Freslon, Adam Rennie, Adam Skalski, and Joachim Zacharias for useful discussions and correspondence. Moreover, they would like to thank the referees for their careful reading and their helpful comments. The first listed author acknowledges support from the EU Staff Exchange project 101086394 \enquote{Operator Algebras That One Can See}. The second listed author was supported by the Swedish Research Council Grant VR 2018-0350.
\subsection*{Notation}Let $F$ and $G$ be real valued functions on some parameter space $Z$. We write $F \lesssim G$ whenever there is a constant $M>0$ such that for all $z\in Z$ we have $F(z) \leq M G(z)$. Similarly, we define the symbol $\gtrsim$ and write $F\simeq G$ if $F \lesssim G$ and $F \gtrsim G$. 

\section{Preliminaries}
\label{prelsec}
In this section we recall topological Markov chains and their Deaconu--Renault groupoids. We equip the latter with a metric-measure space structure reflecting the underlying dynamics.

\subsection{Topological Markov chain}
\label{subsec:topmark}
We first recall the basics on topological Markov chains, also known as subshifts of finite type. The reader can find more on this topic in \cite{Katok_Has,marcuslind}. Topological Markov chains are of importance to hyperbolic dynamical systems as the former are simpler, zero-dimensional and ``symbolically code'' the latter. From an ergodic and measure theoretic perspective much can be said about hyperbolic dynamical systems from the theory of topological Markov chains, see \cite{Katok_Has,marcuslind}. In this subsection, we first describe topological Markov chains and their metric, define their canonical measure and show how the two  fit together into an Ahlfors regular metric-measure space.

Let $A$ be a $N\times N$-matrix with entries in $\{0,1\}$. We will tacitly assume that $A$ is \emph{primitive}, namely that there is a $k\in \mathbb{N}$ such that all entries of $A^k$ are positive. A string $\alpha=\alpha_1\ldots \alpha_n$ with each $\alpha_i\in \{1,\ldots, N\}$ will be called a \emph{word} and its \textit{length} is $|\alpha|:=n$. We call $\{1,\ldots, N\}$ the \textit{alphabet}. The word $\alpha$ is \textit{admissible} if $A_{\alpha_i,\alpha_{i+1}}=1$ for $i=1,\ldots, |\alpha|-1$. We will make frequent use of the empty string $\alpha$ for which $n=0$, which we will call the \emph{empty word} ${\o}$ of length $|{\o}|=0$. We note that the set $V_A$ of admissible finite words together with the empty word form a rooted tree, with root ${\o}$ and a vertex for each admissible word $\alpha$, and an edge going from $\alpha$ to $\alpha j$ if and only if $\alpha j$ is admissible.

We consider the closed subset $\Sigma_A\subset \{1,\ldots,N\}^{\mathbb N}$ defined by 
$$
\Sigma_A=\{x=(x_n)_{n\in \mathbb N}\in \{1,\ldots,N\}^{\mathbb N}: A_{x_n,x_{n+1}}=1\}.
$$
Here $\{1,\ldots,N\}^{\mathbb N}$ is equipped with the product topology induced from the discrete topology on the alphabet $\{1,\ldots, N\}$. By Tychonoff's theorem, $\{1,\ldots,N\}^{\mathbb N}$ is compact and therefore so is $\Sigma_A$. Motivated by the preceding paragraph, we call $\Sigma_A$ the \textit{space of admissible infinite words}. A basis for the topology on $\Sigma_A$ is given by the cylinder sets associated to finite words $\alpha=\alpha_1\ldots \alpha_n$ defined as
$$
C(\alpha):=\{x\in \Sigma_A: x_i=\alpha_i,\, \text{for}\, 1\leq i\leq n\}.
$$
By convention $C(\o):=\Sigma_A$, the cylinder set of the empty word. Note that $C(\alpha)=\emptyset$ unless $\alpha\in V_A$. The cylinder sets $C(\alpha)$ form a neighborhood basis of open compact sets, so $\Sigma_A$ is totally disconnected. Since $A$ is primitive, $\Sigma_A$ has no isolated points. In particular, $\Sigma_A$ is a Cantor space. Geometrically, the reader can think of $\Sigma_A$ as the Gromov boundary of the rooted tree of finite words $V_A$ and the cylinder sets $C(\alpha)$ as the shadow on the Gromov boundary of all rays in $V_A$ going from ${\o}$ to infinity passing through $\alpha$.

The dynamics on the space $\Sigma_A$ is given by the shift map defined by 
$$\sigma_A:\Sigma_A\to \Sigma_A:\quad \sigma_A(x_1x_2x_3\cdots):=x_2x_3\cdots.$$
The map $\sigma_A$ is a surjective local homeomorphism. Indeed, $\sigma_A$ restricts to a homeomorphism from any cylinder set $C(\alpha)$, with $\alpha\in V_A\setminus \{{\o}\}$, to its image. The topological Markov chain associated with the primitive matrix $A$ is the pair $(\Sigma_A,\sigma_A)$ of the Cantor set $\Sigma_A$ equipped with the shift map $\sigma_{A}$. Let us give two examples that we will refer to throughout the paper. 

\begin{example}
\label{onfirst}
If $A$ is the $N\times N$-matrix consisting only of $1$:s, then all words are admissible. As is common in the literature, for this example we shall replace the notation involving $A$ with $N$.  For instance, we have that 
$$V_N=\bigcup_{k=0}^\infty \{1,\ldots, N\}^k, \quad\mbox{and}\quad \Sigma_N=\{1,\ldots, N\}^{\mathbb N}.$$
This example is called the full $N$-shift.
\end{example}

\begin{example}
\label{freefirst}
We now describe the topological Markov chain giving rise to the free group acting on its Gromov boundary. For $d\geq 2$, take generators $a_1,\ldots, a_d$ for the free group $\mathbb F_d$ and consider the alphabet given by the symmetric generating set $S=\{a_1,a_1^{-1},\ldots, a_d,a_d^{-1}\}$. Further, let $A=(A_{xy})_{x,y\in S}$ be the $2d\times 2d$-matrix given by 
$$A_{xy}:=
\begin{cases}
1, \; &\mbox{if}\; x\neq y^{-1},\\
0, \; &\mbox{if}\; x= y^{-1}.
\end{cases}$$
We note that an admissible word is nothing else than a reduced word in the set $S$, so the graph of finite admissible words for this matrix coincides with the Cayley graph of the free group $\mathbb{F}_d$ on the $d$ generators $a_1,\ldots, a_d$. Consequently, for this example
$$\Sigma_A=\partial \mathbb{F}_d,$$
the Gromov boundary of $\mathbb{F}_d$. Moreover, the orbits of the group action of $\mathbb{F}_d$ on its Gromov boundary $\partial \mathbb{F}_d$ coincide with a weak version of tail equivalence classes for the dynamics of $\sigma_A$ on $\Sigma_A$, see Example \ref{freethird}.
\end{example}

We now turn to describing the metric structure on the topological Markov chain $(\Sigma_A,\sigma_A)$. Given $\lambda >1$, we equip $\Sigma_A$ with the ultrametric
$$
\mathrm{d}(x,y):=\lambda^{-\inf \{n-1: x_n\neq y_n\}},
$$
with the convention that $\inf \varnothing = \infty$. In geometric terms, $\mathrm{d}(x,y)=\lambda^{-k}$ precisely when there is a word $\alpha$ of length $k$ such that $x=\alpha x'$ and $y=\alpha y'$, for $x'$ and $y'$ two infinite admissible words starting with different letters. We note that for every $x\in \Sigma_A$, the $\mathrm{d}$-ball
$$
B(x,\lambda^{-n})=C(x_1\ldots x_n).
$$
In particular, for $n>1$ the shift map restricts to a homeomorphism 
$$B(x,\lambda^{-n})\to B(\sigma_A(x),\lambda^{-n+1}).$$ 
The well known fact that the topological Markov chain $(\Sigma_A,\sigma_A)$ is a locally expanding dynamical system now follows. 

The canonical measure on the topological Markov chain $(\Sigma_A,\sigma_A)$ is the Parry measure, which is defined as follows. Let $\lambda_{\max}$ be the Perron--Frobenius eigenvalue of $A$. This is the largest eigenvalue of $A$ whose eigenvector $u$ has positive entries. Since $A$ is a primitive matrix we have $\lambda_{\max}>1$. We write $v$ for the Perron--Frobenius eigenvector of the transpose matrix $A^{T}$. We assume that $u,v$ are normalised so that $\sum_{i=1}^{N}u_iv_i=1$. We set $p_i=u_iv_i$, $i=1,\ldots,N$. The distribution $p=(p_1,\ldots, p_N)$ on the finite set $\{1,\ldots, N\}$ induces the Parry measure $\mu$ on $\Sigma_A$. Note that $\mu$ is $\sigma_A$-invariant. Moreover, if $\alpha=\alpha_1\ldots \alpha_n$ is an admissible finite word then 
\begin{equation*}
\mu(C(\alpha))= \frac{v_{\alpha_1}u_{\alpha_n}}{\lambda_{\max}^{n-1}}.
\end{equation*}

\begin{example}
\label{onsecond}
For the full $N$-shift (see Example \ref{onfirst}) we have $A^k=N^{k-1}A$, so Beurling's formula gives $\lambda_{\max}=N$. Moreover, since $u_i=v_i=N^{-1/2}$, the Parry measure is given by 
$$\mu(C(\alpha))=N^{-|\alpha|}, \quad \alpha \in V_N.$$
\end{example}

\begin{example}
\label{freesecond}
For the Gromov boundary $\partial \mathbb{F}_d$ (see Example \ref{freefirst}), we find $\lambda_{\max}=2d-1$ with the corresponding eigenvectors having entries $u_i=v_i=(2d)^{-1/2}$. In particular, the Parry measure is given by 
$$\mu(C(\alpha))=\frac{1}{2d(2d-1)^{|\alpha|-1}}, \quad \alpha \in \mathbb{F}_d\setminus \{1\}.$$
Here we identify the unit $1\in \mathbb{F}_d$ with the empty word in the graph of finite words.
\end{example}

The coordinates of $u,v$ are all positive, so we conclude that there is some $C\geq 1$ so that for every $x\in \Sigma_A$ and $0\leq r \leq 1$,
\begin{equation}\label{eq:Ahlfors_reg}
C^{-1}r^{\delta}\leq \mu(B(x,r))\leq Cr^{\delta}, \quad\mbox{for}\quad \delta=\log_{\lambda}(\lambda_{\max}).
\end{equation}
In other words, the measure $\mu$ is Ahlfors $\delta$-regular. The smallest $C$ satisfying \eqref{eq:Ahlfors_reg} is called the Ahlfors constant of the metric-measure space $(\Sigma_{A},\mathrm{d},\mu)$. Note that $\log(\lambda_{\max})>0$ is the topological entropy of $(\Sigma_A,\sigma_A)$. For more details see \cite{Katok_Has}. Moreover, by \cite[Lemma 2.5]{gatto}, for every $s>0$ it holds that 
\begin{equation}\label{eq:Ahlfors_int}
\int_{B(x,r)} \frac{1}{\d(x,y)^{\delta -s}} \d \mu (y) \simeq r^s.
\end{equation}

\subsection{Deaconu--Renault groupoid} 
\label{subsec:deac}
We use the groupoid approach to Cuntz--Krieger algebras \cite{deaacaneudoaod,Ren1,Ren2}, mainly presenting their structure as in \cite{GMCK,GMON,GMR}. The reader can find more details on groupoids and their $C^*$-algebras in \cite{Ren1}. We consider only {\'e}tale groupoids. The basic idea is to replace a discrete group $\Gamma$ acting on a space $X$ with a structure encoding both the group and the space. At a set theoretical level a groupoid $G$ consists of the following data: a set of arrows that by abuse of notation denote $G$, a unit space $G^{(0)}$ with source $s:G\to G^{(0)}$, range $r:G\to G^{(0)}$, unit $\iota:G^{(0)}\to G$, inversion $(\cdot){}^{-1}:G\to G$, and finally a multiplication $$m:G {}_{\hspace{0.1cm}s}\times_r G\to G,$$
satisfying a series of axioms mimicking a group acting on a space. We sometimes write a groupoid $G$ with base $G^{(0)}$ as $G\rightrightarrows G^{(0)}$. If $G$ and $G^{(0)}$ are locally compact Hausdorff spaces and the maps $s$, $r$ and $\iota$ are local homeomorphisms we say that $G$ is an {\'e}tale groupoid. If this is the case, $C_c(G)$ forms a $*$-algebra in the convolution product 
\begin{equation}
\label{eq:conv}
f_1\star f_2(g)=\sum_{m(g_1,g_2)=g} f_1(g_1)f_2(g_2),
\end{equation}
and the star operation $f^*(g)=\overline{f(g^{-1})}$. The reduced $C^*$-algebra generated by $C_c(G)$ is denoted by $C^*_{r}(G)$ and contains $C_{c}(G)$ as a dense subalgebra.

In the context of a topological Markov chain $(\Sigma_{A},\sigma_{A})$, we consider  the {\'e}tale groupoid 
$$
G_A=\{(x,n,y)\in \Sigma_A\times \mathbb Z \times \Sigma_A: \exists k\geq 0\,\, \text{such that}\,\,n+k\geq 0 \,\,\text{and}\,\,  \sigma_A^{n+k}(x)=\sigma_A^{k}(y)\}\rightrightarrows \Sigma_A.
$$
The source and range maps are 
$$
s(x,n,y):=y,\qquad r(x,n,y):=x,
$$
and the partial multiplication and inversion are given by 
$$m((x,n,y),(y,\ell,z))\equiv(x,n,y)(y,\ell,z):=(x,n+\ell,z),\quad (x,n,y)^{-1}:=(y,-n,x).$$ 
Already at this stage, we see how the groupoid $G_A$ encodes the dynamics of a topological Markov chain by keeping track of its orbits. In particular, $G_A$ encodes some sort of weak tail equivalence.

Further, consider the maps $\kappa:G_A\to \mathbb N\cup \{0\}$ and $c:G_A\to \mathbb Z$ defined as 
\begin{align*}
\kappa(x,n,y)&:=\min \left\{k\geq \max \{0,-n\}: \sigma_A^{n+k}(x)=\sigma_A^k(y)\right\},\\
c(x,n,y)&:=n.
\end{align*}
We equip $G_A$ with the coarsest topology that makes the maps $s,r,\kappa,$ and $c$ continuous. As such, $G_{A}$ is a locally compact, totally disconnected \'etale groupoid. The $C^*$-algebra $C^{*}_{r}(G_A)$ is the Cuntz--Krieger algebra of $A$, see Theorem \ref{renthm} below. The map $c:G_A\to\mathbb{Z}$ is a homomorphism and the formula
\begin{equation}
\label{eq:cocycle-action}
\alpha_{z}(f)(x,n,y):=z^{n}f(x,n,y),\quad f\in C_{c}(G_{A}), \,\, z\in\mathbb{T},\end{equation}
defines an action of the group $\mathbb{T}$ of unit complex numbers by $*$-automorphisms of the $C^{*}$-algebra $C^{*}_{r}(G_{A})$ (see \cite[Section II.5]{Ren1}).

A basis for the topology of $G_{A}$ is given by the bisections indexed by admissible finite words $\alpha$, $\beta$
$$C(\alpha,\beta):=\{(x,|\alpha|-|\beta|,y): x\in C(\alpha), y\in C(\beta), \sigma_A^{|\alpha|}(x)=\sigma_A^{|\beta|}(y)\}.$$

We shall decompose $G_A$ over a more refined collection of bisections than $C(\alpha,\beta)$. To this end, we first observe that for $g=(x,n,y)\in G_A$ one has $\sigma_A^{n+\kappa(g)}(x)=\sigma_A^{\kappa(g)}(y)$ and we recall the following important interplay between the maps $\kappa$ and $c$.

\begin{lemma}[{\cite[Proof of Proposition 5.1.4]{GMCK}}]
\label{lem:GMinteraction}
For every $(x,n,y)\in G_A$ it holds that 
$$\kappa(x,n,y)=
\begin{cases}
\kappa(\sigma_A(x),n-1,y), & \text{if } \kappa(x,n,y)+n\geq 1\\
\kappa(\sigma_A(x),n-1,y)-1, & \text{if } \kappa(x,n,y)+n=0.
\end{cases}
$$
\end{lemma}
We now introduce our main tool for decomposing $G_{A}$.

\begin{definition}
\label{def: discr}
The \textit{discretization map} $\ell$ is defined to be the map
$$\ell=(\ell_r,\ell_s):G_A\to V_A\times (V_A\setminus \{{\o}\}),$$ 
that takes $g=(x,n,y)$ to the pair of words $(\ell_r(g),\ell_s(g))$ with $|\ell_r(g)|=n+\kappa(g)$ and $|\ell_s(g)|=\kappa(g)+1$ such that 
$$x=\ell_r(g)\sigma_A^{\kappa(g)}(y) \quad\mbox{and}\quad y=\ell_s(g)\sigma_A^{\kappa(g)+1}(y).$$
\end{definition}

The image of $\ell$ consists exactly of pairs $(\alpha,\beta)\in V_A\times (V_A\setminus \{{\o}\})$ such that either $\alpha={\o}$ or $A_{\alpha_{|\alpha|},\beta_{|\beta|}}=1$. We also note that $\ell$ is continuous since it is constant on the open subsets of each $C(\alpha,\beta)$ where the value of $\kappa$ is fixed. To simplify notation, we write
\begin{equation}\label{eq:indexset}
I_A:=\{\gamma=\alpha.\beta\in V_A\times (V_A\setminus \{{\o}\}):\; \alpha={\o}\;\mbox{or}\; A_{\alpha_{|\alpha|},\beta_{|\beta|}}=1\; \}.
\end{equation}
We think of the notation $\gamma=\alpha.\beta$ as $\gamma$ consisting of a two-sided finite word with the dot $.$ indicating the midpoint of the word. Geometrically, $\gamma=\alpha.\beta$ should be thought of as two strands $\alpha$ and $\beta$ in the rooted tree $V_A$ joined at the root and whose endpoints are compatible in the sense that $A_{\alpha_{|\alpha|},\beta_{|\beta|}}=1$. For $\gamma=\alpha.\beta\in I_A$, we write
$$r(\gamma):=\alpha\quad\mbox{and}\quad s(\gamma):=\beta.$$

For $\gamma=\alpha.\beta\in I_A$, we define 
\begin{equation*}
G_{\gamma}:=\ell^{-1}(\gamma)\equiv \left\{(x,n,y)\in G_A: \begin{matrix} x=\alpha\sigma_A^{|\beta|-1}(y), \; y=\beta\sigma_A^{|\beta|}(y),\; \mbox{and}\\ 
n=|\alpha|-|\beta|+1, \; \kappa(x,n,y)=|\beta|-1\end{matrix} \right\}.
\end{equation*}
Note that, in view of Lemma \ref{lem:GMinteraction}, $G_{\gamma}$ is always non-empty. Since $\ell$ is continuous, we arrive at a decomposition into clopen subsets 
\begin{equation}
\label{eq:G_decomp}
G_A=\bigsqcup_{\gamma\in I_A} G_{\gamma}.
\end{equation}

By construction, we have that 
\begin{equation*}
\kappa|_{G_\gamma}=|s(\gamma)|-1 \quad\mbox{and}\quad c|_{G_\gamma}=|r(\gamma)|-|s(\gamma)|+1.
\end{equation*}
Moreover, the following holds. 

\begin{lemma}\label{prop:source maps}
For each $\gamma\in I_A$ the clopen subset $G_\gamma\subseteq G_A$ is a bisection and the restricted source map 
$$s_{\gamma}:=s|_{G_\gamma}:G_{\gamma}\to C(s(\gamma)),$$
and restricted range map 
$$r_{\gamma}:=r|_{G_\gamma}:G_{\gamma}\to C(r(\gamma)),$$
are homeomorphisms. Moreover, the product on $G_A$ yields the partially defined maps
$$G_{\alpha.\beta}\times G_{\beta.\beta'}\to G_{\alpha\beta_{|\beta|}.\beta'},$$
and the inversion map satisfies $$G_{\alpha.\beta}^{-1}=G_{\widehat{\beta}.\alpha \beta_{|\beta|}},$$ where $\widehat{\beta}\in V_A$ is uniquely determined by $\beta=\hat{\beta}\beta_{|\beta|}$.
\end{lemma}

For the sequel it is important to understand how the shift map acts on \eqref{eq:G_decomp}.

\begin{lemma}\label{lem:shift_inv}
Let $i\in \{1,\ldots, N\}$ and let $\gamma=\alpha.\beta \in I_A$ be such that $i\gamma:=i\alpha.\beta\in I_A$. Then, the map 
$$\sigma_{i\gamma}:G_{i\gamma}\to G_{\gamma},\quad (x,n,y)\mapsto (\sigma_A(x),n-1,y),$$ 
is a homeomorphism and satisfies 
$$s_{\gamma}\circ \sigma_{i\gamma}= s_{i\gamma}.$$
Moreover, if $\gamma={\o}.\beta$ with $\beta=\hat{\beta}\beta_{|\beta|}$ and $\hat{\beta}\neq {\o}$, then for $\hat{\gamma}:={\o}.\hat{\beta}\in I_A$ the map 
\begin{align*}
\sigma_{\gamma}:G_{\gamma}&\to G_{\hat{\gamma}}, \\
& (\sigma_A^{|\beta|-1}(x),-|\beta|+1,x)\mapsto (\sigma_A^{|\beta|-2}(x),-|\beta|+2,x)=(\sigma_A^{|\hat{\beta}|-1}(x),-|\hat{\beta}|+1,x),
\end{align*}
is a homeomorphism onto its image.
\end{lemma}

\subsection{Metric-measure space structure on $G_A$}
For every $\gamma \in I_A$, we define the pull-back metric $\mathrm{d}_{\gamma}$ on $G_{\gamma}$ by 
$$\mathrm{d}_{\gamma}(g_1,g_2):=\mathrm{d}(s_{\gamma}(g_1),s_{\gamma}(g_2)).$$ 
In other words, the metric $\mathrm{d}_\gamma$ on $G_\gamma$ can be defined from declaring $s_\gamma:G_\gamma\to C(s(\gamma))$ to be an isometry. Then, we define an extended metric $\mathrm{d}_{G_{A}}$ on $G_A$, generating the topology on $G_A$, by 
$$\mathrm{d}_{G_A}(g_1,g_2):=
\begin{cases}
\mathrm{d}_{\gamma}(g_1,g_2), & \text{if  for some $\gamma\in I_A$ both}\; g_1,g_2\in G_{\gamma}\\
\infty, & \text{if else}
\end{cases}.
$$

\begin{remark}
\label{blablaon}
The extended metric defined above differs from the one used in the preceding work \cite{GMON} of the second and third listed author. In \cite{GMON}, $G_A$ was decomposed in terms of $G_{\alpha,k}$ for $\alpha\in V_A$ and $k\geq 0$, that with our current notation corresponds to $\bigsqcup_{|\beta|=k+1}G_{\alpha.\beta},$ with $\beta$ satisfying $\alpha . \beta \in I_A$. Then, an extended metric was defined using this decomposition and declaring $s|_{G_{\alpha,k}}$ to be isometric. As we discuss below in Remark \ref{rem:Unbounded_com}, this extended metric led to various issues in obtaining spectral triples from log-Laplacians. 

\end{remark}

We define the pull-back finite Borel measure $\mu_{\gamma}:=s_\gamma^*\mu$ on $G_{\gamma}$. For an open $B\subset G_{\gamma}$, $\mu_\gamma$ is defined by 
$$\mu_{\gamma}(B)= \mu(s_{\gamma}(B)).$$ 
The collection $\{\mu_{\gamma}\}_{\gamma\in I_A}$ gives a Borel measure $\mu_{G_A}$ on $G_A$. Now since each $s_{\gamma}(G_{\gamma})=C(s(\gamma))$ is clopen in $\Sigma_A$, and for the latter $\mu$ is Ahlfors $\delta$-regular, we obtain the following.

\begin{lemma}
\label{ojnljnjonad}
For every $\gamma \in I_A$, the metric-measure space $(G_{\gamma},\mathrm{d}_{\gamma},\mu_{\gamma})$ is Ahlfors $\delta$-regular with $\delta= \log_{\lambda}(\lambda_{\max}).$
\end{lemma}

From Lemma \ref{lem:shift_inv} we obtain the next result describing shift invariance in the collection $\{\mu_{\gamma}\}_{\gamma \in I_A}$.

\begin{lemma}
\label{prop:shift_inv_measure}
For every $i\in \{1,\ldots, N\}$, $\gamma \in I_A$ such that $i\gamma \in I_A$ and open set $B\subset G_{i\gamma}$, one has 
$$\mu_{i\gamma}(B)=\mu_{\gamma}(\sigma_{i\gamma}(B)).$$
\end{lemma}

\section{$O_{A}$-representation and wavelet basis}

The Cuntz--Krieger algebra $O_{A}$, as originally defined in \cite{CK}, is the universal $C^{*}$-algebra generated by $N$ elements $S_{i}$, $i\in\{1,\cdots, N\}$, subject to the relations
\begin{equation}\label{eq: CK-relations} 
\sum_{i=1}^{N}S_{i}S_{i}^{*}=1 \quad \mbox{and}\quad  S_{i}^{*}S_{k}=\delta_{ik}\sum_{j=1}^{N}A_{ij}S_{j}S_{j}^{*}.
\end{equation}
The universal $C^{*}$-algebra $O_{A}$ carries an action of the group $\mathbb{T}$ by $*$-automorphisms, known as the \emph{gauge action} and determined by the formula 
\begin{equation}
\label{eq:gauge-action}
\alpha_{z}(S_{i}):=zS_{i},\quad z\in\mathbb{T},\,\,i\in\{1,\cdots, N\}.
\end{equation} 
The Cuntz--Krieger algebra $O_A$ can be equivalently described as the reduced groupoid $C^*$-algebra $C_r^*(G_A)$ of the Deaconu--Renault groupoid introduced in the Subsection \ref{subsec:deac}. The isomorphism intertwines the gauge action \eqref{eq:gauge-action} and the action \eqref{eq:cocycle-action}. 
The following theorem can be found in the literature, see \cite{Ren1, Ren2}.

\begin{thm}
\label{renthm}
Let $A$ be a primitive matrix of $0$:s and $1$:s. The Cuntz--Krieger algebra  $O_A$ is isomorphic to the groupoid $C^{*}$-algebra $C_r^*(G_A)$ and under this isomorphism each generator $S_i$ is given by the characteristic function of the set
$$\{(x,1,\sigma_A(x)):x\in C(i)\}.$$
Moreover, the isomorphism is equivariant for the respective $\mathbb{T}$-actions.
\end{thm}

The $C^*$-algebra $C^*_{r}(G_A)$ can be constructed explicitly as the norm closure of the $*$-algebra $C_{c}(G_{A})$ in its representation on $L^2(G_A,\mu_{G_A})$ by convolution operators via Equation \eqref{eq:conv}. For our purposes, it is important to note that the $*$-subalgebra $\LC\subset C_{c}(G_{A})$ of locally constant, compactly supported functions is dense in $C^{*}_{r}(G_{A})$. 

In the sequel we will often identify the standard Cuntz--Krieger generators $S_i$ for $i\in \{1,\ldots, N\}$ of $O_A$ with the characteristic functions from Theorem \ref{renthm}. In this section we shall study how $L^2(G_A,\mu_{G_A})$ decomposes under the decomposition \eqref{eq:G_decomp} and how $O_A$ acts on this decomposition.

\begin{remark}
\label{cyclicremark}
Let us make a remark about inclusions and representations to avoid confusion. We first note that since $A$ is primitive, there is a unique KMS-state $\phi$ on $O_A$ defined on the dense subspace $C_c(G_A)$ by 
$$\phi(f)=\int_{\Sigma_A} f(x,0,x)\mathrm{d}\mu,$$
where $\mu$ denotes the Parry measure. In particular, the identity map on $C_c(G_A)$ extends to an isometric identification of the $O_A$-representations
$$L^2(G_A,\mu_A)=L^2(O_A,\phi),$$
where $L^2(O_A,\phi)$ denotes the GNS-representation associated to $\phi$. We can therefore view $O_A$ and its dense subalgebras $\LC$, $C_c(G_A)$ as dense subspaces of $L^2(G_A,\mu_A)$. The precise inclusion is implemented by acting on the cyclic vector $1\in O_A\subseteq L^2(O_A,\phi)$. 
\end{remark}

\begin{example}
\label{onthird}
For the full $N$-shift (see Example \ref{onfirst} and \ref{onsecond}), we have that $O_N\simeq C_r^*(G_N)$ is the Cuntz algebra. It is the universal unital $C^*$-algebra generated by $N$ isometries with orthogonal ranges.
\end{example}

\begin{example}
\label{freethird}
For the group action of $\mathbb{F}_d$ on its Gromov boundary $\partial \mathbb{F}_d$ (see Example \ref{freefirst} and \ref{freesecond}), we have a groupoid isomorphism $G_A\simeq \partial \mathbb{F}_d\rtimes \mathbb{F}_d$. This is proved in \cite[Section 2]{spielberg}, see also \cite[Subsection 3.4.3]{GMCK}. In particular, for the free group $O_A\simeq C(\partial \mathbb{F}_d)\rtimes \mathbb{F}_d$. The same argument shows that $L^2(G_A,\mu_{G_A})\simeq L^2(\partial \mathbb{F}_d,\mu)\otimes \ell^2(\mathbb{F}_d)$ compatibly with the natural covariant representation of $C(\partial \mathbb{F}_d)\rtimes \mathbb{F}_d$.
\end{example}

\subsection{Cuntz--Krieger algebra} 

For every $f\in L^2(G_A, \mu_{G_A})$ it holds that 
$$(S_i  f)(x,n,y) =
\begin{cases}
0, & \text{if } x \not\in C(i)\\
f(\sigma_A(x),n-1,y), & \text{if else}
\end{cases}.
$$

To determine the support of functions $S_i f$ we shall use the following result.

\begin{lemma}\label{lem:inverse_images}
Let $\gamma=\alpha.\beta\in I_A$ and $(x,n,y)\in G_A$ with $(\sigma_A(x),n-1,y)\in G_{\gamma}$. Then,
\begin{enumerate}
\item if $\alpha \neq {\o}$ or $|\beta|=1$, we have that $(x,n,y)\in \bigsqcup_{j:j\gamma \in I_A} G_{j\gamma}$;
\item if $\alpha = {\o}$ and $|\beta|> 1$, we have that 
$$(x,n,y)\in \left(\bigsqcup_{j: j\gamma\in I_A} G_{j\gamma}\right)\bigsqcup G_{{\o}.\hat{\beta}},$$ 
where $\hat{\beta}:=\beta_1\ldots \beta_{|\beta|-1}$. 
\end{enumerate}
\end{lemma}

We now present each generator $S_i$ as a matrix operator acting on the decomposition 
\begin{equation*}\label{eq:L2_decomp}
L^2(G_A, \mu_{G_A})=\bigoplus_{\gamma\in I_A}L^2(G_{\gamma},\mu_{\gamma}).
\end{equation*}
Consider the projections $P_{\gamma}:L^2(G_A, \mu_{G_A})\to L^2(G_A, \mu_{G_A})$ onto $L^2(G_{\gamma},\mu_{\gamma})$. We write 
$$P_\gamma S_iP_{\gamma'}:L^2(G_{\gamma'},\mu_{\gamma'})\to L^2(G_{\gamma},\mu_{\gamma}),$$
for the $\left(\gamma,\gamma'\right)$-entry of $S_i$. In fact, for $f\in L^2(G_A,\mu_{G_A})$ and $(x,n,y)\in G_A$ we have that 
$$(P_\gamma S_iP_{\gamma'}f)(x,n,y)=
\begin{cases}
f(\sigma_A(x),n-1,y), & \text{if } x\in C(i),\,\, (x,n,y)\in G_{\gamma},\,\, (\sigma_A(x),n-1,y)\in G_{\gamma'}\\
0, & \text{otherwise}
\end{cases}.$$

Given that all metric-measure information on $G_A$ is defined in terms of the isometric homeomorphism
$$\hat{s}:=\bigsqcup_{\gamma\in I_A} s_\gamma:G_A=\bigsqcup_{\gamma\in I_A} G_{\gamma}\to \bigsqcup_{\gamma\in I_A} C(s(\gamma))\subseteq \bigsqcup_{\gamma\in I_A} \Sigma_A,$$
we have a unitary isomorphism that by abuse of notation we denote
\begin{equation}\label{eq:univ_isom}
\hat{s}:L^2(G_A, \mu_{G_A})=\bigoplus_{\gamma\in I_A}L^2(G_{\gamma},\mu_{\gamma})\xrightarrow{} \bigoplus_{\gamma\in I_A}L^2(C(s(\gamma)),\mu).
\end{equation}
To get a better understanding of $P_\gamma S_iP_{\gamma'}$, we identify it under this unitary with the operator
$$S_{i,\gamma}^{\gamma'}:L^2(C(s(\gamma')),\mu)\to L^2(C(s(\gamma)),\mu),$$
and $S_i$ with the matrix operator $(S_{i,\gamma}^{\gamma'})_{\gamma,\gamma'\in I_A}$.

\begin{prop}
\label{descmatrdecomp}
Let $i\in \{1,\ldots, N\}$. Then, every column of the matrix operator $S_i=(S_{i,\gamma}^{\gamma'})_{\gamma,\gamma'\in I_A}$ has at most two non-zero entries and every row at most $N$ non-zero entries. The entry $S_{i,\gamma}^{\gamma'}$ of $S_i$ as an operator $L^2(C(s(\gamma')),\mu)\to L^2(C(s(\gamma)),\mu)$ can be described by
$$S_{i,\gamma}^{\gamma'} f(s(\gamma)x):=
\theta_i(\gamma,\gamma')f(s(\gamma')x), $$
where $\theta_i:I_A\times I_A\to \{0,1\}$ is the characteristic function of the set
$$\left\{(\gamma,\gamma')\in  I_A\times I_A: \quad
\begin{matrix} s(\gamma)=s(\gamma')\,\, \text{and}\,\, \,r(\gamma)=ir(\gamma') \\
\mbox{or}\\
r(\gamma)=r(\gamma')={\o}\,\, \text{and}\,\, s(\gamma)=\widehat{s(\gamma')}\,\, \text{and}\,\,  s(\gamma')_{|s(\gamma')|-1}=i
\end{matrix}\right\}.$$

In particular, $S_{i,\gamma}^{\gamma'}=0$ when
\begin{enumerate}
\item $(r(\gamma')\neq \o$ or $|s(\gamma')|=1)$ and $(\gamma\neq i\gamma'$ or $i\gamma' \notin I_A)$, 
\item $r(\gamma') = {\o}$ and $|s(\gamma')|> 1$: 
\begin{enumerate}
\item $A_{i,s(\gamma')_{|s(\gamma')|}}=1$ and $(\gamma\neq i.\gamma'$ or $\gamma\neq  {\o}.\widehat{s(\gamma')})$,
\item $A_{i,s(\gamma')_{|s(\gamma')|}}=0$,
\item $\gamma=\widehat{\gamma'}$ and $s(\gamma')_{|s(\gamma')|-1}\neq i$.
\end{enumerate}
\end{enumerate}
\end{prop}

\subsection{The wavelet basis for $L^2(G_A,\mu_{G_A})$}
\label{lknlnaldjnaljn}

Let us turn to constructing a wavelet basis for $L^2(G_A, \mu_{G_A})$ by utilizing the unitary isomorphism \eqref{eq:univ_isom}. The construction is similar, yet different as we see in Remark \ref{rem:diff_wavelets}, to those in \cite{marcollipaol}. Specifically, for a finite word $\beta\in V_A\setminus \{{\o}\}$ we aim to construct an orthonormal basis $H(\beta)$ for $L^2(C(\beta),\mu)$. 

To this end, we decompose $L^2(C(\beta),\mu)$ into finite-dimensional spaces as
$$L^2(C(\beta),\mu)=\bigoplus_{k=0}^\infty L^2_k(C(\beta),\mu),$$
by declaring $\bigoplus_{k=0}^N L^2_k(C(\beta),\mu)$ to consist of the linear span of $\{\chi_{C(\beta\nu)}: |\nu|=N\}$ in $L^2(C(\beta),\mu)$. The space $L^2_0(C(\beta),\mu)$ is one-dimensional and is spanned by the unit vector 
$$h_{\beta}:=\mu(C(\beta))^{-1/2} \chi_{C(\beta)}.$$
We write $H_0(\beta)=\{h_{\beta}\}$ and for $k>0$, $H_k(\beta)$ will denote an orthonormal basis of $L^2_k(C(\beta),\mu)$ with the property that for $h\in H_k(\beta)$:
\begin{enumerate}
\item There exists a unique $\omega_\beta(h)\in V_A$ starting with $\beta$ of length $|\beta|+k-1\geq 1$ such that $h$ is supported in $C(\omega_\beta(h))$, that we sometimes denote by $C(h)$.
\item $h$ is constant on $C(\mu)$ for any $|\mu|\geq |\beta|+k$.
\end{enumerate}
The orthonormal basis $H(\beta)$ is defined as $\bigsqcup_{k\geq 0} H_k(\beta)$. Note that for $h\in H(\beta)\setminus H_0(\beta)$,
$$\int_{C(\beta)} h\d\mu=\int_{C(\omega_\beta(h))} h\d\mu=0,$$
since the left hand side is $\mu(C(\beta))^{1/2}\langle h_\beta,h\rangle$. We shall enumerate $H(\beta)$ as $$(h_\beta,h_{\nu,j})_{\nu=\beta\nu_0\in V_A, j\in J_\beta(\nu)}$$ where $\omega_\beta(h_{\nu,j})=\nu$ and $J_\beta(\nu)\subseteq \{1,\ldots, N\}$ is a non-empty index set such that $$\# J_\beta(\nu)= \# \{i\in \{1,\ldots, N\}:A_{\nu_{|\nu|},i}=1\}-1\geq 1.$$

\begin{example}
\label{onfourth}
For the full $N$-shift (see Example \ref{onfirst}, \ref{onsecond} and \ref{onthird}), we can take the basis $H(\beta)=(h_\beta, h_{\nu,j})_{ \nu=\beta\nu_0\in V_N, j\in J_\beta(\nu)}$ as follows. We have 
$$h_\beta=N^{|\beta|/2}\chi_{C(\beta)}.$$
Also, for $|\nu|\geq |\beta|$ we can choose $J_\beta(\nu)= \{1,\ldots, N-1\}$ and for $j\in  J_\beta(\nu)$ we have a basis element
$$h_{\nu,j}=N^{|\nu|/2}\sum_{k=1}^{N} \mathrm{e}^{\frac{2 \pi i j(k-1)}{N}}\chi_{C(\nu k)}\in H_{|\nu|-|\beta|+1}(\beta).$$
\end{example}

\begin{example}
\label{freefourth}
For the Gromov boundary $\partial \mathbb{F}_d$ (see Example \ref{freefirst}, \ref{freesecond} and \ref{freethird}), we can take the basis $H(\beta)=(h_\beta, h_{\nu,j})_{\nu=\beta\nu_0\in \mathbb{F}_d, j\in J_\beta(\nu)}$ as follows. We have
$$h_\beta=(2d)^{1/2}(2d-1)^{(|\beta|-1)/2}\chi_{C(\beta)}.$$ 
Recall the alphabet $S$ from Example \ref{freefirst}. For $|\nu|\geq |\beta|$ we can choose $J_\beta(\nu)= S\setminus \{\nu_{|\nu|},\nu_{|\nu|}^{-1}\}$ that we enumerate as $J_\beta(\nu)=\{1,\ldots, 2d-2\}$ and identify $\nu_{|\nu|}$ with $2d-1$. For $j\in  J_\beta(\nu)$ we have a basis element
$$h_{\nu,j}=(2d)^{1/2}(2d-1)^{(|\nu|-1)/2}\sum_{k=1}^{2d-1} \mathrm{e}^{\frac{2 \pi i j(k-1)}{2d-1}}\chi_{C(\nu k)}\in H_{|\nu|-|\beta|+1}(\beta).$$
\end{example}

We can now obtain a wavelet basis for $L^2(G_A, \mu_{G_A})$. The action of the generators on this basis can be described using  Lemmas \ref{lem:shift_inv} and \ref{lem:inverse_images}.

\begin{thm}
\label{lknlknadasd}
Define the set 
$$\mathfrak{I}_A:=\{(\gamma,\nu,j)\in I_A\times V_A\setminus \{{\o}\}\times \{1,\ldots, N\}: \; \nu=s(\gamma)\nu_0, \; j\in J_{s(\gamma)}(\nu)\}.$$
Then the collection $(\mathrm{e}_\gamma,\mathrm{e}_{(\gamma,\nu,j)})_{(\gamma,\nu,j)\in \mathfrak{I}_A}\subset \LC$ defined from 
$$\mathrm{e}_{\gamma}:=\chi_{G_\gamma}\cdot (h_{s(\gamma)}\circ s),$$ where $\chi_{G_\gamma}$ is the characteristic function of $G_{\gamma}\subset G_A$, 
and 
$$\mathrm{e}_{(\gamma,\nu,j)}:=\chi_{G_\gamma}\cdot (h_{\nu,j}\circ s)$$
is an orthonormal basis for $L^2(G_A,\mu_{G_A})$. Moreover, consider a generator $S_i$ for some $i\in \{1,\ldots ,N\}$. If $i\gamma \not\in I_A$ then $S_i\mathrm{e}_{(\gamma,\nu,j)}=0.$ Also, if $i\gamma \in I_A$ and
\begin{enumerate}[(a)]
\item $r(\gamma)\neq {\o}$ or $|s(\gamma)|=1$, then $$S_i\mathrm{e}_{(\gamma,\nu,j)}=\mathrm{e}_{(i\gamma,\nu,j)},$$
\item $r(\gamma)={\o}$ and $|s(\gamma)|>1$ and $s(\gamma)_{|s(\gamma)|-1}\neq i$, then $$S_i\mathrm{e}_{(\gamma,\nu,j)}=\mathrm{e}_{(i\gamma,\nu,j)},$$
\item $r(\gamma)={\o}$ and $|s(\gamma)|>1$ and $s(\gamma)_{|s(\gamma)|-1}=i$, then $$S_i\mathrm{e}_{(\gamma,\nu,j)}(x,n,y)=\begin{cases} \mathrm{e}_{(i\gamma,\nu,j)}(x,n,y), & \text{if } (x,n,y)\in G_{i\gamma},\\ h_{\nu,j}\circ s|_{\sigma_{\gamma}(G_{\gamma})}(x,n,y), & \text{if } (x,n,y)\in \sigma_{\gamma}(G_{\gamma}),\\
0, & \text{if else}
\end{cases},$$
where $\sigma_{\gamma}(G_{\gamma})\subset G_{\widehat{\gamma}}$ as described in Lemma \ref{lem:shift_inv}.
\end{enumerate}
\end{thm}

\begin{example}
\label{ex:ONgroupoiddecomp}
For the Cuntz algebra $O_N$ we can identify 
$$\mathfrak{I}_N=I_N\times (V_N\setminus\{{\o}\}) \times \{1,\ldots, N-1\}.$$
In this index set, the basis is given by 
$$
\mathrm{e}_\gamma:= N^{|s(\gamma)|/2}\chi_{G_\gamma}, \quad\mbox{and}\quad
\mathrm{e}_{\gamma,\nu,j}:=N^{|\nu|/2}\sum_{k=1}^{N} \mathrm{e}^{\frac{2 \pi i j(k-1)}{N}}\chi_{G_\gamma}\cdot (\chi_{C(\nu k)} \circ s).$$

\end{example}

\begin{remark}\label{rem:diff_wavelets}
We note that there is a difference from our construction to that in \cite{marcollipaol}. In \cite{marcollipaol}, one finds a construction for a basis of $L^2(\Sigma_A,\mu)$ and its relation to a representation of $O_A$ on $L^2(\Sigma_A,\mu)$. While the basis we consider originates in the constructions of \cite{marcollipaol}, we note that the left action of the Cuntz--Krieger algebra $O_A$ differs substantially on the two spaces. Indeed, there exists no isometry $L^2(\Sigma_A,\mu)\hookrightarrow L^2(G_A,\mu_{G_A})$ commuting with the left $O_A$-action. Such an isometry can not exist since  $1\in L^2(\Sigma_A,\mu)$ is a cyclic vector such that $S_i^*1=0$ and no such vector exists in $L^2(G_A,\mu_{G_A})$.
\end{remark}

\section{The logarithmic Dirichlet Laplacian on $G_A$}

We have that 
$$L^2(G_A, \mu_{G_A})=\bigoplus_{\gamma\in I_A}L^2(G_{\gamma},\mu_{\gamma}).$$ 
Therefore, we can define the positive, self-adjoint operator $\Delta$ acting on $L^2(G_A, \mu_{G_A})$ as 
$$\Delta=\bigoplus_{\gamma\in I_A}\Delta_{\gamma},$$ 
where each $\Delta_{\gamma}$ is the logarithmic Dirichlet Laplacian on $L^2(G_{\gamma},\mu_{\gamma})$ with its Ahlfors regular metric-measure space structure $(G_{\gamma},\mathrm{d}_\gamma,\mu_{\gamma})$ discussed in Proposition \ref{ojnljnjonad}. We can identify $\Delta_\gamma$ with $P_{\gamma}\Delta P_{\gamma}$. Also, each $\Delta_{\gamma}$ is essentially self-adjoint on locally constant functions. We can deduce this fact from the direct computation appearing below in Lemma \ref{complemlog} which implies that $\Delta_\gamma$ can be diagonalised by the Haar wavelets of Subsection \ref{lknlnaldjnaljn}. In particular, the space $\LC$ of locally constant functions with compact support on $G_A$ is a core for $\Delta$. Often the support of elements in $\LC$ will be in some $G_{\gamma}$ or a finite union of those.

Note that for $f\in \LC$ the operator $\Delta$ admits the integral representation 
\begin{equation}\label{eq:Int_rep}
\Delta f(g_1)=\int_{G_A} \frac{f(g_1)-f(g_2)}{\mathrm{d}_{G_A}(g_1,g_2)^\delta} \d \mu_{G_A}(g_2),
\end{equation}
with singular kernel $\mathrm{d}_{G_A}(g_1,g_2)^{-\delta}$. The importance of such integral representations is explained in more detail in \cite{GerMes}, see also \cite{GMON,GoffUsa,MesHecke}. 

\begin{remark}
\label{lknlknlknad}
The singular kernel $\mathrm{d}_{G_A}(g_1,g_2)^{-\delta}$ is a localised version of the singular kernel used in \cite{GMON}. Recall that in \cite{GMON}, another extended metric was used as discussed above in Remark \ref{blablaon}. If we write $\tilde{\mathrm{d}}_{G_A}$ for the metric used in \cite{GMON}, we have $\mathrm{d}_{G_A}(g_1,g_2)\geq \tilde{\mathrm{d}}_{G_A}(g_1,g_2)$. In particular, $\mathrm{d}_{G_A}(g_1,g_2)^{-\delta}\leq \tilde{\mathrm{d}}_{G_A}(g_1,g_2)^{-\delta}$ while the reversed inequality is not always true. Nevertheless, the kernels have similar singularity since for $g_1$ and $g_2$ sufficiently close (assuming $\tilde{\mathrm{d}}_{G_A}(g_1,g_2)^{-\delta}\geq \lambda_{\max}^{k+1}$ to be precise) the kernels coincide: $\mathrm{d}_{G_A}(g_1,g_2)^{-\delta}=\tilde{\mathrm{d}}_{G_A}(g_1,g_2)^{-\delta}$.
\end{remark}

\subsection{Spectral properties}
\label{specsubsec}

The spectral properties of $\Delta$ can be determined explicitly as in \cite[Section 5]{GerMes}. First we consider the situation on each $G_\gamma$, or equivalently, on each cylinder set $C(\beta)$ for $\beta\in V_A\setminus \{{\o}\}$. Also, for any $i,j\in \{1,\ldots, N\}$ denote by $P_{i,j}$ the transition probability of going from \textit{vertex} $i$ to \textit{vertex} $j$ in the finite directed graph associated to $A$. Then, from \cite[p. 175]{Katok_Has} we see that $$P_{i,j}=\frac{A_{i,j}u_j}{\lambda_{\max}u_i}.$$

\begin{lemma}
\label{complemlog}
Let $\beta\in V_A\setminus \{{\o}\}$. The orthonormal basis $H(\beta)=(h_\beta,h_{\nu,j})_{\nu=\beta\nu_0\in V_A, j\in J_\beta(\nu)}$ constructed in Subsection \ref{lknlnaldjnaljn} is an eigenbasis for the logarithmic Dirichlet Laplacian $\Delta_{C(\beta)}$ on $L^2(C(\beta),\mu)$. We have that 
$$\ker\Delta_{C(\beta)}=\mathbb{C} h_\beta$$
and 
$$\Delta_{C(\beta)}h_{\nu,j}=\lambda^A_\beta(\nu)h_{\nu,j},$$
where if $C(\nu)=C(\beta)$ then $\lambda^A_\beta(\nu)=1$, and if $C(\nu)\subsetneq C(\beta)$ then
$$\lambda_\beta^A(\nu)=1+\lambda_{\max}v_{\nu_1}\sum_{k=0}^{|\nu|-|\beta|-1}u_{\nu_{|\nu|-k-1}}(1-P_{\nu_{|\nu|-k-1}, \nu_{|\nu|-k}}).
$$ 
\end{lemma}

The computations proving Lemma \ref{complemlog} closely mimicks that of \cite[Proposition 5.6]{GerMes}, which in turn is inspired by \cite[Theorem 3.1]{actisaim}.

\begin{proof}
The equality $\ker\Delta_{C(\beta)}=\mathbb{C} h_\beta$ follows from \cite[Proposition 3.4]{GerMes}. By the argument in \cite[Theorem 3.1]{actisaim} we also get
$$\Delta_{C(\beta)}h_{\nu,j}(x)=0, \quad\mbox{for $x\in C(\beta)\setminus C(\nu)$}.$$
Moreover, for $x\in C(\nu)$ we have that 
$$\Delta_{C(\beta)}h_{\nu,j}(x)=\left(1+\int_{C(\beta)\setminus C(\nu)} \mathrm{d}(x,y)^{-\delta}\d\mu(y)\right)h(x).$$
If $C(\nu)=C(\beta)$ then $\Delta_{C(\beta)}h_{\nu,j}=h_{\nu,j}$. On the other hand, if $C(\nu)\subsetneq C(\beta)$ then 
$$C(\beta)\setminus C(\nu)=\bigsqcup_{k=0}^{|\nu|-|\beta|-1} \left[ B(x,\lambda^{-|\nu|+k+1})\setminus B(x,\lambda^{-|\nu|+k})\right].$$
Therefore, 
\begin{align*}
\int_{C(\beta)\setminus C(\nu)} \mathrm{d}(x,y)^{-\delta}\d\mu(y)=&\sum_{k=0}^{|\nu|-|\beta|-1} \int_{B(x,\lambda^{-|\nu|+k+1})\setminus B(x,\lambda^{-|\nu|+k})}  \mathrm{d}(x,y)^{-\delta}\d\mu(y)=\\
=&\sum_{k=0}^{|\nu|-|\beta|-1} \lambda^{\delta(|\nu|-k-1)}\left(\frac{v_{\nu_1}u_{\nu_{|\nu|-k-1}}}{\lambda_{\rm max}^{|\nu|-k-2}}-\frac{v_{\nu_1}u_{\nu_{|\nu|-k}}}{\lambda_{\rm max}^{|\nu|-k-1}}\right)=\\
=&\sum_{k=0}^{|\nu|-|\beta|-1} v_{\nu_1}\left(u_{\nu_{|\nu|-k-1}}\lambda_{\max}-u_{\nu_{|\nu|-k}}\right)\\
=&\lambda_{\max}v_{\nu_1}\sum_{k=0}^{|\nu|-|\beta|-1}u_{\nu_{|\nu|-k-1}}(1-P_{\nu_{|\nu|-k-1}, \nu_{|\nu|-k}}).
\end{align*}
Here we used $\lambda^\delta=\lambda_{\rm max}$ by the definition of $\delta$.
\end{proof}

\begin{example}\label{onfifth}
For the full $N$-shift (see Example \ref{onfirst}, \ref{onsecond}, \ref{onthird} and \ref{onfourth}, \ref{ex:ONgroupoiddecomp}), we have that when $C(\nu)\subsetneq C(\beta)$ then $$\lambda_\beta^A(\nu)=1+(1-N^{-1})(|\nu|-|\beta|).$$
\end{example}

\begin{example}\label{freefifth}
For the Gromov boundary $\partial \mathbb{F}_d$ (see Example \ref{freefirst}, \ref{freesecond}, \ref{freethird} and \ref{freefourth}), whenever $C(\nu)\subsetneq C(\beta)$ then $$\lambda_\beta^A(\nu)=1+(1-d^{-1})(|\nu|-|\beta|).$$
\end{example}

We shall now obtain estimates of the eigenvalues of each $\Delta_{C(\beta)}$. For this we require the following observation.

\begin{lemma}\label{lem:monochains}
Let $\nu\in V_A$ with $|\nu|> 3(N+1)$ and $\# \{i\in \{1,\ldots, N\}: A_{\nu_{|\nu|},i}=1\}\geq 2.$ Then, it holds that $$\# \{i\in \{1,\ldots, |\nu|-1\}: P_{\nu_i,\nu_{i+1}}\neq 1\}>\frac{|\nu|}{3(N+1)}.$$
\end{lemma}

\begin{proof}
Consider such $\nu\in V_A$ and write $|\nu|=k(N+1)+m$, for $k\geq 1$ and $0\leq m\leq N$. Also, for every $1\leq l\leq k$ denote by $\nu(l)$ the $l$-th consecutive subword of $\nu$ with $|\nu(l)|=N+1$, that is, $\nu(l)=\nu_{(l-1)(N+1)+1}\ldots \nu_{l(N+1)}$. 

We claim that for each $\nu(l)$ there is some $j\in \{(l-1)(N+1)+1, \ldots, l(N+1)\}$ such that the letter $\nu(l)_j$ has
\begin{equation}\label{eq:monochains}
\# \{i\in \{1,\ldots, N\}: A_{\nu(l)_j,i}=1\}\geq 2.
\end{equation}
It then follows that $$\# \{i\in \{1,\ldots, |\nu|-1\}: P_{\nu_i,\nu_{i+1}}\neq 1\}\geq k-1>\frac{|\nu|}{N+1}-2\geq \frac{|\nu|}{3(N+1)}.$$ We have bounded below by $k-1$ instead of $k$ because it can happen that $m=0$ and we have already assumed that the last subword $\nu(k)$ contains a letter satisfying \eqref{eq:monochains}, that is, the last letter $\nu_{|\nu|}$.

To prove the claim \eqref{eq:monochains} assume to the contrary that there is some $\nu(l)$ such that for every $j\in \{(l-1)(N+1)+1, \ldots, l(N+1)\}$ one has 
\begin{equation}\label{eq:monochains2}
\# \{i\in \{1,\ldots, N\}: A_{\nu(l)_j,i}=1\}=1.
\end{equation}
Now $\nu(l)$ contains a cycle since $|\nu(l)|>N$ and due to \eqref{eq:monochains2} this cycle has to be repeated indefinitely. In particular, the last letter $\nu_{|\nu|}$ satisfies \eqref{eq:monochains2}, a contradiction.
\end{proof}

Lemma \ref{lem:monochains} can now be applied to any $\nu \in V_A$ that appears in the orthonormal basis $H(\beta)=(h_\beta,h_{\nu,j})_{\nu=\beta\nu_0\in V_A, j\in J_\beta(\nu)}$, and since $A$ is primitive and every coordinate of the eigenvectors $u,v$ is positive, we obtain the following. 

\begin{prop}
\label{asjnkjnada}
There are constants $0<C_0\leq C_1 $ so that for every $\beta \in V_A\setminus \{{\o}\}$ and every eigenvalue $\lambda_\beta^A(\nu)$ of Lemma \ref{complemlog} it holds that $$C_0(|\nu|-|\beta|)\leq \lambda_\beta^A(\nu)\leq C_1(|\nu|-|\beta|).$$
\end{prop}

Finally, we obtain the spectral decomposition of $\Delta$ on $L^2(G_A,\mu_A)$.

\begin{thm}
\label{kljnkjnadadjbnojbn}
The orthonormal basis $(\mathrm{e}_{\gamma},\mathrm{e}_{(\gamma,\nu,j)})_{(\gamma,\nu,j)\in \mathfrak{I}_A}\subset \LC$  for $L^2(G_A,\mu_{G_A})$ from Theorem \ref{lknlknadasd} consists of eigenfunctions of $\Delta$. More precisely, $\ker(\Delta)$ is the closed linear span of $(\mathrm{e}_{\gamma})_{\gamma\in I_A}$ and for $(\gamma,\nu,j)\in \mathfrak{I}_A$, 
$$\Delta \mathrm{e}_{(\gamma,\nu,j)}=\lambda_{s(\gamma)}^A(\nu)\mathrm{e}_{(\gamma,\nu,j)}.$$
\end{thm}

\subsection{Heat kernels for the full $N$-shift and the algebra $O_N$}

Let us describe the heat kernels arising from the full $N$-shift and at the groupoid level for the algebra $O_N$. The algebra $O_N$ arises from the case that $A$ is a matrix only consisting of $1$:s, see Example \ref{onfirst}, \ref{onsecond}, \ref{onthird}, \ref{onfourth}, \ref{ex:ONgroupoiddecomp} and \ref{onfifth} above. We start by describing the heat kernel of each $\Delta_\gamma$ and for the purpose of obtaining a geometry relevant object on the groupoid, we compute the heat kernel of $\Delta + M_{\ell}$ on $L^2(G_N,\mu_{G_N})$, where $M_{\ell}$ is the multiplication operator by the proper continuous function $\ell:G_N\to \mathbb N$ defined as
\begin{equation}\label{eq:scale}
\ell|_{G_\gamma}:=|\gamma|=|r(\gamma)|+|s(\gamma)|.
\end{equation}
Note that $M_{\ell}$ is positive and essentially self-adjoint on ${\rm C_{c}^\infty}(G_N)$. Moreover, for $s>0$ and $\gamma\in I_N$, we define the Riesz potential operator $R_{\gamma,s}$ on $L^2(G_\gamma,\mu_{G_N})$ by 
$$R_{\gamma,s}f(g_1):=\int_{G_\gamma} \frac{f(g_2)}{\mathrm{d}(g_1,g_2)^{\delta-s}}\d \mu(g_2).$$
Riesz potential operators are an indispensable tool for studying metric-measure spaces \cite{heinmono,Zah}. We should highlight that the limit $s\to 0$ is highly singular as captured by the logarithmic Dirichlet Laplacian. 

\begin{prop}
\label{lnlknkojnad}
For every $\gamma \in I_N$ and $t>0$, the heat operator $e^{-t\Delta_{\gamma}}$ admits a kernel $k_{\gamma,t}\in L^1(G_{\gamma}\times G_{\gamma},\mu_{\gamma}\times \mu_{\gamma})$. Specifically, for $g_1\neq g_2 \in G_{\gamma}$ it is given by 
$$k_{\gamma,t}(g_1,g_2)=h_\gamma(t)+H_\gamma(t)\mathrm{d}_\gamma(g_1,g_2)^{-\delta+\delta't},$$
where $\delta'=(1-N^{-1})\log_{\lambda}(\mathrm{e})$ and
\begin{align*}
h_\gamma(t)&=N^{|s(\gamma)|}\left(1-\frac{(N-1)\mathrm{e}^{-t}}{N\mathrm{e}^{-t(1-N^{-1})}-1}\right)\\
H_\gamma(t)&=\frac{N-N\mathrm{e}^{-t(1-N^{-1})}}{N\mathrm{e}^{-t(1-N^{-1})}-1}\mathrm{e}^{-t(1-(1-N^{-1})|s(\gamma)|)}.
\end{align*}
Further, $k_{\gamma,t}(g,g)<\infty$ if and only if $t>\delta/\delta'$, and for $t>\delta/\delta'$ we have that $k_{\gamma,t}$ is continuous. Finally, 
\begin{equation}
\label{etdeltaon}
\mathrm{e}^{-t\Delta_\gamma}=h_\gamma(t)N^{-|s(\gamma)|}P_{\Delta_\gamma}+H_\gamma(t)R_{\gamma,\delta't},
\end{equation}
where $P_{\Delta_\gamma}$ denotes the projection onto the kernel of $\Delta_\gamma$. 
\end{prop}

We note that the largest eigenvalue of the $N\times N$-matrix consisting only of $1$:s is $N$, and so
$$\frac{\delta}{\delta'}=\frac{\log_\lambda(N)}{\log_\lambda(\mathrm{e})(1-N^{-1})}=\frac{\log(N)}{1-N^{-1}}.$$
In particular, both $h_\gamma$ and $H_\gamma$ are smooth for $t>\delta/\delta'$. Both $h_\gamma$ and $H_\gamma$ extend meromorphically in $t$ to the complex plane with a simple pole at $t=\delta/\delta'$.

\begin{proof}
We can equivalently perform our computations on $C(\beta)$ for $\beta=s(\gamma)$ and denote the kernel by $k_{\beta,t}$. We compute using the wavelet basis $H(\beta)=(h_\beta, h_{\nu,j})_{\nu=\beta\nu_0\in V_N, j\in \{1,\ldots, N-1\}}$ from Example \ref{onfourth}. By Lemma \ref{complemlog} (see also \cite[Proof of Proposition 5.6]{GerMes}) we have that $\Delta_{C(\beta)}h_\beta=0$ and
$$\Delta_{C(\beta)}h_{\beta\nu_0,j}=\left(1+(1-N^{-1})|\nu_0|\right)h_{\beta\nu_0,j}.$$
Take $x\neq y\in C(\beta)$ and set $k=-\log_\lambda \mathrm{d}(x,y)\geq |\beta|$, so there is a finite word $\tilde{\nu}\in V_N$ of length $k$ such that $\tilde{\nu}=\beta \tilde{\nu}_0$ and $x=\tilde{\nu}x_0$ and $y=\tilde{\nu}y_0$ where $x_0$ and $y_0$ differ in their first letter that we call $p_x$ and $p_y$. 

First, assume that $k>|\beta|$. Then, 
\scriptsize
\begin{align*}
k_{\beta,t}(x,y)=&h_\beta(x)h_\beta(y)+\\
&+\sum_{\nu_0\in V_N}\sum_{j=1}^{N-1}\sum_{k_1,k_2=1}^{N}N^{|\beta|+|\nu_0|}\mathrm{e}^{-t(1+(1-N^{-1})|\nu_0|)} \mathrm{e}^{\frac{2 \pi i j (k_1-k_2)}{N}}\chi_{C(\beta\nu_0 k_1)}(x) \chi_{C(\beta\nu_0 k_2)}(y)=\\
=&N^{|\beta|}+N^{|\beta|}(N-1)e^{-t}\sum_{l=0}^{k-|\beta|-1}N^{l}\mathrm{e}^{-t(1-N^{-1})l}+\sum_{j=1}^{N-1}N^{|\beta|+|\tilde{\nu}_0|}\mathrm{e}^{-t(1+(1-N^{-1})|\tilde{\nu}_0|)} \mathrm{e}^{\frac{2 \pi i j (p_x-p_y)}{N}}\\
=&N^{|\beta|}+N^{|\beta|}(N-1)e^{-t}\frac{\left(Ne^{-t(1-N^{-1}}\right)^{k-|\beta|}-1}{Ne^{-t(1-N^{-1})}-1}-N^k\mathrm{e}^{-t(1+(1-N^{-1})(k-|\beta|))}\\
{}^{(\ast)}=&N^{|\beta|}\left(1-\frac{(N-1)\mathrm{e}^{-t}}{N\mathrm{e}^{-t(1-N^{-1})}-1}\right)+\left(\frac{N-N\mathrm{e}^{-t(1-N^{-1})}}{N\mathrm{e}^{-t(1-N^{-1})}-1}\mathrm{e}^{-t(1-(1-N^{-1})|\beta|)}\right)\left(N\mathrm{e}^{-t(1-N^{-1})}\right)^k.
\end{align*}
\normalsize
To complete the calculation for $k>|\beta|$ notice that $$\left(N\mathrm{e}^{-t(1-N^{-1})}\right)^k=\mathrm{d}(x,y)^{-\log_\lambda(N)+t(1-N^{-1})\log_\lambda(e)}.$$ Now, if $k=|\beta|$, the second term 
\begin{equation}\label{eq:noterm}
N^{|\beta|}(N-1)e^{-t}\sum_{l=0}^{k-|\beta|-1}N^{l}\mathrm{e}^{-t(1-N^{-1})l}
\end{equation}
in the second equality will not appear and hence we get
\begin{align*}
k_{\beta,t}(x,y)=&N^{|\beta|}(1-e^{-t})\\
{}^{(\ast \ast)}=&N^{|\beta|}-\mathrm{e}^{-t(1-(1-N^{-1})|\beta|)}\mathrm{d}(x,y)^{-\log_\lambda(N)+t(1-N^{-1})\log_\lambda(e)}.
\end{align*}
Observe now that the expression $(\ast \ast)$ can be derived from $(\ast)$ by putting $k=\beta$, even though \eqref{eq:noterm} is not defined. The rest properties of $k_{\beta,t}$ are easy to check.

To now conclude Equation \eqref{etdeltaon}, we note that the integral kernel of $P_{\Delta_\gamma}$ is the constant function $(g_1,g_2)\mapsto N^{|s(\gamma)|}$. Here the constant $N^{|s(\gamma)|}$ arises from the normalization $\|1\|_{L^2(C(s(\gamma)))}^2=\mu(C(s(\gamma)))=N^{-|s(\gamma)|}$. In particular, the computation above shows that for $f\in L^2(G_\gamma,\mu_\gamma)$,
\begin{align*}
\mathrm{e}^{-t\Delta_\gamma}f(g_1)=&~\int_{G_\gamma}\left(h_\gamma(t)+H_\gamma(t)\mathrm{d}_\gamma(g_1,g_2)^{-\delta+\delta't}\right)f(g_2)\mathrm{d}\mu_\gamma(g_2)\\
=&~\left( h_\gamma(t)N^{-|s(\gamma)|}P_{\Delta_\gamma}+H_\gamma(t)R_{\gamma,\delta't}\right)f(g_1).
\end{align*}
This completes the proof.
\end{proof}

As a corollary, we obtain that the one parameter families $t\mapsto H\left(\frac{t}{\delta'}\right)R_{\gamma,t}$ form a semigroup up to rank one operators.

\begin{cor}
For every $\gamma \in I_N$, the family $\{R_{\gamma,s}\}_{s>0}$ of Riesz potentials satisfies the composition rule
\begin{align*}
H_\gamma\left(\frac{s_1}{\delta'}\right)H_\gamma&\left(\frac{s_2}{\delta'}\right)R_{\gamma,s_1}R_{\gamma,s_2}=H_{\gamma}\left(\frac{s_1+s_2}{\delta'}\right)R_{\gamma,s_1+s_2}+\\
&+\left(N^{-|s(\gamma)|}h_{\gamma}(\frac{s_1+s_2}{\delta'})+h_{\gamma}(\frac{s_1}{\delta'})h_{\gamma}(\frac{s_2}{\delta'})N^{-|s(\gamma)|}-h_{\gamma}(\frac{s_1}{\delta'}) -h_{\gamma}(\frac{s_2}{\delta'})\right) P_{\Delta_{\gamma}}.
\end{align*}
\end{cor}
\begin{proof}
For every $\gamma \in I_N$ and $s>0$ it holds that $[P_{\Delta_{\gamma}},R_{\gamma,s}]=0$ and $$R_{\gamma,s}P_{\Delta_{\gamma}}=\frac{1-h_{\gamma}(\frac{s}{\delta'})N^{-|s(\gamma)|}}{H_{\gamma}(\frac{s}{\delta'})}.$$ Then, the semigroup property of $\{e^{-s\Delta_{\gamma}}\}_{s>0}$  and an algebraic manipulation gives us the result.
\end{proof}
By summing the result of Proposition \ref{lnlknkojnad} over all components, using the estimate \eqref{eq:Ahlfors_int}, and taking into account the function $\ell:G_N\to \mathbb N$ in \eqref{eq:scale}, we arrive at the following.

\begin{thm}
\label{heatthm}
For every $t>\frac{\delta}{\delta'}$, the heat operator $\mathrm{e}^{-t(\Delta+M_{\ell})}$ admits a kernel 
$$K_t\in L^1(G_N\times G_N,\mu_{G_N}\times \mu_{G_N})\cap C(G_N\times G_N).$$ 
Specifically, for $g_1\neq g_2\in G_N$, it is given by 
$$K_t(g_1,g_2)=\begin{cases}e^{-t|\gamma|}k_{\gamma,t}(g_1,g_2), & \text{if  }\; g_1,g_2\in G_{\gamma}\\
0, & \text{otherwise.}\end{cases}$$ 
\end{thm}

\subsection{The commutator with $O_A$}
We now arrive at the main technical result towards the construction of a spectral triple from the logarithmic Dirichlet Laplacian $\Delta$.
\begin{prop}\label{prop:bounded_com_Lap}
For $i\in \{1,\ldots, N\}$ the commutator $[\Delta, S_i]$ extends to a bounded operator on $L^2(G_A, \mu_{G_A})$. Moreover, $\|[\Delta,S_{i}]\|\leq C$ where $C\geq 1$ is the Ahlfors constant as in \eqref{eq:Ahlfors_reg}. In particular, for any $a\in \LC$, $[\Delta, a]$ extends to a bounded operator on $L^2(G_A, \mu_{G_A})$.
\end{prop}

\begin{proof}
The commutator $[\Delta,S_i]$ is a matrix operator with entries 
$$P_{\gamma}[\Delta,S_i]P_{\gamma'}=\Delta_{\gamma}P_{\gamma}S_iP_{\gamma'}-P_{\gamma}S_iP_{\gamma'}\Delta_{\gamma'},$$ 
since $\Delta$ is diagonal in the decomposition of $L^2(G_A,\mu_{G_A})$. We claim that each entry extends to a bounded operator and the operator norms of those entries are uniformly bounded. Also, from Proposition \ref{descmatrdecomp} we have that every column of $[\Delta,S_i]$ has at most two non-zero entries and every row at most $N$ non-zero entries. Therefore, from Schur's test $[\Delta,S_i]$ is bounded.  Moreover, in our calculations it suffices to work with locally constant functions, due to the self-adjointness of each $\Delta$ and since convoluting with $S_i$ preserves locally constant functions. 

It follows from Proposition \ref{descmatrdecomp} that $P_{\gamma}[\Delta,S_i]P_{\gamma'}=0$ unless we are in one of the two cases 
\begin{enumerate}
\item $s(\gamma)=s(\gamma')$ and $r(\gamma)=ir(\gamma')$, or 
\item $r(\gamma)=r(\gamma')={\o}$ and $s(\gamma)=\widehat{s(\gamma')}$ and $s(\gamma')_{|s(\gamma')|-1}=i$.
\end{enumerate}

We first show that case (1), that is, $\gamma=i\gamma'\in I_A$, contributes nothing to the commutator. Indeed, for $f\in \LC$ and $(x,n,y)\in G_A\setminus G_{i\gamma'}$ it holds that $P_{i\gamma'}[\Delta,S_i]P_{\gamma'}f(x,n,y)=0$, while for $(x,n,y)\in G_{i\gamma'}$ we can write
\begin{align*}
P_{i\gamma'} S_iP_{\gamma'}\Delta_{\gamma'}f(x,n,y)&= S_i\Delta_{\gamma'}f(x,n,y)\\
&= \Delta_{\gamma'}f(\sigma_A(x),n-1,y)\\
&= \int_{G_{\gamma'}} \frac{f(\sigma_A(x),n-1,y)-f(z,m,w)}{\mathrm{d}_{\gamma'}((\sigma_A(x),n-1,y), (z,m,w))^{\delta}} \d \mu_{\gamma'}(z,m,w)\\
&=\int_{G_{\gamma'}} \frac{f(\sigma_A(x),n-1,y)-f(z,n-1,w)}{\mathrm{d}(y,w)^{\delta}} \d \mu_{\gamma'}(z,n-1,w)
\end{align*}
and 
\begin{align*}
\Delta_{i\gamma'}P_{i\gamma'} S_iP_{\gamma'}f(x,n,y)&=\Delta_{i\gamma'}S_iP_{\gamma'}f(x,n,y)\\
&= \int_{G_{i\gamma'}}\frac{S_i P_{\gamma'}f(x,n,y)-S_i P_{\gamma'}f(z,m,w)}{\mathrm{d}_{i\gamma'}((x,n,y),(z,m,w))^{\delta}}\d \mu_{i\gamma'}(z,m,w)\\
&= \int_{G_{i\gamma'}}\frac{f(\sigma_A(x),n-1,y)-(f\circ \sigma_{i\gamma'})(z,n,w)}{\mathrm{d}_{i\gamma'}((x,n,y),(z,n,w))^{\delta}}\d \mu_{i\gamma'}(z,n,w)\\
&= \int_{G_{\gamma'}}\frac{f(\sigma_A(x),n-1,y)-f(z',n',w')}{\mathrm{d}_{i\gamma'}((x,n,y), (iz', n'+1, w'))^{\delta}}\d \mu_{\gamma'}(z', n', w')\\
&=\int_{G_{\gamma'}}\frac{f(\sigma_A(x),n-1,y)-f(z,n-1,w)}{\mathrm{d}(y,w)^{\delta}}\d \mu_{\gamma'}(z,n-1,w).
\end{align*}
The fourth equality comes from changing $(z,n,w)\mapsto (\sigma_{i\gamma'})^{-1}(z',n',w') = (iz', n'+1, w')$ and Proposition \ref{prop:shift_inv_measure}. We conclude that 
$$\Delta_{i\gamma'}P_{i\gamma'}S_iP_{\gamma'}=P_{i\gamma'}S_iP_{\gamma'}\Delta_{\gamma'},$$
and that case (1) contributes nothing to the commutator.

We claim that the non-zero entries in the commutator $[\Delta,S_i]$ comes from case (2). In other words, $[\Delta,S_i]$ is supported on the entries $P_{{\o}.\hat{\beta}}[\Delta,S_i]P_{{\o}.\beta}$ for $\beta$ of length $k+1\geq 2$ and $\beta_k=i$, with norms uniformly bounded in $k$. 

Take an $f\in \LC$. Observe that $P_{{\o}.\hat{\beta}}[\Delta,S_i]P_{{\o}.\beta}f$ vanishes on $G_A\setminus G_{{\o}.\hat{\beta}}$. Let $(\sigma_A^{k-1}(x),-k+1,x)\in G_{{\o}.\hat{\beta}}$ be arbitrary. Then, we have that 
\begin{align*}
&\Delta_{{\o}.\hat{\beta}}P_{{\o}.\hat{\beta}}S_iP_{{\o}.\beta}f(\sigma_A^{k-1}(x),-k+1,x)\\
&=\int_{G_{{\o}.\hat{\beta}}}\frac{S_i P_{{\o}.\beta}f(\sigma_A^{k-1}(x),-k+1,x)-S_i P_{{\o}.\beta}f(\sigma_A^{k-1}(y),-k+1,y)}{\mathrm{d}(x,y)^{\delta}}\d \mu_{{\o}.\hat{\beta}}(\sigma_A^{k-1}(y),-k+1,y)\\
&=\int_{G_{{\o}.\hat{\beta}}}\frac{P_{{\o}.\beta}f(\sigma_A^k(x),-k,x)-P_{{\o}.\beta}f(\sigma_A^k(y),-k,y)}{\mathrm{d}(x,y)^{\delta}}\d \mu_{{\o}.\hat{\beta}}(\sigma_A^{k-1}(y),-k+1,y)\\
\end{align*}
and 
\begin{align*}
&P_{{\o}.\hat{\beta}}S_iP_{{\o}.\beta}\Delta_{{\o}.\beta}f(\sigma_A^{k-1}(x),-k+1,x)\\
&=\Delta_{{\o}.\beta}f(\sigma_A^{k}(x),-k,x)\\
&=\int_{G_{{\o}.\beta}}\frac{P_{{\o}.\beta}f(\sigma_A^k(x),-k,x)-P_{{\o}.\beta}f(\sigma_A^k(y),-k,y)}{\mathrm{d}(x,y)^{\delta}}\d \mu_{{\o}.\beta}(\sigma_A^k(y),-k,y)\\
&=\int_{\sigma_{{\o}.\beta}(G_{{\o}.\beta})}\frac{P_{{\o}.\beta}f(\sigma_A^k(x),-k,x)-P_{{\o}.\beta}f(\sigma_A^k(z),-k,z)}{\mathrm{d}(x,z)^{\delta}}\d \mu_{{\o}.\hat{\beta}}(\sigma_A^{k-1}(z),-k+1,z),
\end{align*}
where the last equality comes from 
$$(\sigma_A^k(y),-k,y)\mapsto (\sigma_{{\o}.\beta})^{-1}((\sigma_A^{k-1}(z),-k+1,z))=(\sigma_A^{k}(z),-k,z).$$
Therefore, our computation is concluded from noting that
\begin{align*}
P_{{\o}.\hat{\beta}}[\Delta,S_i]P_{{\o}.\beta}&f(\sigma_A^{k-1}(x),-k+1,x)=\\
=P_{{\o}.\beta}&f(\sigma_A^k(x),-k,x)\int_{G_{{\o}.\hat{\beta}}\setminus \sigma_{{\o}.\beta}(G_{{\o}.\beta})}\frac{1}{\mathrm{d}(x,y)^{\delta}}\d \mu_{{\o}.\hat{\beta}}(\sigma_A^{k-1}(y),-k+1,y).
\end{align*}
However, 
 $$\int_{G_{{\o}.\hat{\beta}}\setminus \sigma_{{\o}.\beta}(G_{{\o}.\beta})}\frac{1}{\mathrm{d}(x,y)^{\delta}}\d \mu_{{\o}.\hat{\beta}}(\sigma_A^{k-1}(y),-k+1,y)\leq C,$$ and so 
$$\|P_{{\o}.\hat{\beta}}[\Delta,S_i]P_{{\o}.\beta}\|\leq C.$$ 
The proof is now complete.
\end{proof}

\begin{remark}
For full $N$-shift we have $C=1$, hence $\|[\Delta,S_i]\|\leq 1.$
\end{remark}

\begin{remark}\label{rem:Unbounded_com}
We now describe why we have altered the constructions of \cite{GMON} as discussed in Remarks \ref{blablaon} and \ref{lknlknlknad}. In \cite{GMON}, $G_A$ was decomposed into $G_{\alpha,k}$ that here is
$$G_{\alpha,k}=\bigsqcup_{|\beta|=k+1}G_{\alpha.\beta},$$ with $\beta$ being such that $\alpha . \beta \in I_A$. This decomposition and the associated distance functions produce a logarithmic Dirichlet Laplacian whose commutator with each $S_i$ {\bf does not} extend to a bounded operator, in contrast to the statement of \cite{GMON}. Let us prove this statement. For this reason consider the decomposition $$L^2(G_A,\mu_{G_A})=\bigoplus_{\alpha,k}L^2(G_{\alpha,k},\mu_{\alpha,k})$$ with the associated projections $P_{\alpha,k}$. Define the positive, self-adjoint operator $\Delta'$ acting on $L^2(G_A,\mu_A)$ as $$\Delta'=\bigoplus_{\alpha,k}\Delta_{\alpha,k},$$ where each $\Delta_{\alpha,k}$ is the logarithmic Dirichlet Laplacian on $L^2(G_{\alpha,k},\mu_{\alpha,k})$. 

It can be shown that the number of non-zero entries of each column and row of the matrix operator $[\Delta',S_i]$ is uniformly bounded, and each non-zero entry extends to a bounded operator. However, the norms of those non-zero entries will not be uniformly bounded, meaning that $[\Delta',S_i]$ cannot extend to a bounded operator on $L^2(G_A,\mu_{G_A})$. Specifically, let $k\geq 1$, $f\in \LC$ and we focus on the entry $P_{{\o},k-1}[\Delta',S_i]P_{{\o},k}$ which vanishes on $G_A\setminus G_{{\o},k-1}$. 

Let $(\sigma_A^{k-1}(x),-k+1,x)\in G_{{\o},k-1}$ be arbitrary. If $\sigma_A^{k-1}(x)\in C(i)$, changing the variables with the homeomorphism $G_{{\o},k-1}\to G_{{\o},k}$ given by $(\sigma_A^{k-1}(x),-k+1,x)\mapsto (\sigma_A^{k}(x),-k,x)$, gives that $$P_{{\o},k-1}[\Delta',S_i]P_{{\o},k}f(\sigma_A^{k-1}(x),-k+1,x)=0.$$ However, if $\sigma_A^{k-1}(x)\not\in C(i)$, then 
\begin{align*}
&P_{{\o},k-1}[\Delta',S_i]P_{{\o},k}f(\sigma_A^{k-1}(x),-k+1,x)\\
&=-\int_{G_{{\o},k-1}}\frac{S_iP_{{\o},k}f(\sigma_A^{k-1}(y),-k+1,y)}{\d (x,y)^{\delta}}\d \mu_{{\o},k-1}(\sigma_A^{k-1}(y),-k+1,y)\\
&=-\int_{G_{{\o},k-1},\,\,y\in \sigma_A^{-k+1}(C(i))}\frac{P_{{\o},k}f(\sigma_A^{k}(y),-k,y)}{\d (x,y)^{\delta}}\d \mu_{{\o},k-1}(\sigma_A^{k-1}(y),-k+1,y).
\end{align*}
In particular, if $f$ is the characteristic function of $G_{{\o},k}$, the latter integral has absolute value $$\int_{\sigma_A^{-k+1}(C(i))}\frac{1}{\d (x,y)^{\delta}}\d \mu(y)\gtrsim k.$$ In that case, $$\|P_{{\o},k-1}[\Delta',S_i]P_{{\o},k}f\|_{L^2}\gtrsim k (1-\mu(C(i)))^{1/2}$$ and as $\|f\|_{L^2}=1$ since $\mu_{{\o},k}(G_{{\o},k})=1$, we have that $$\|P_{{\o},k-1}[\Delta',S_i]P_{{\o},k}\|\gtrsim k (1-\mu(C(i)))^{1/2}.$$
We conclude that $[\Delta',S_i]$ is not bounded. 
\end{remark}

\section{The noncommutative geometry of $G_A$ as a metric-measure groupoid}\label{sec:NCG of G_A}
We now assemble the spectral triples from $\Delta$ and describe their $K$-homology classes following ideas of \cite{GMCK}. We show that the approach taken in \cite{GMON} for the algebras $O_{N}$ applies to all Cuntz--Krieger algebras $O_{A}$, thus obtaining a new class of spectral triples for these.
\subsection{The Fock space}

We introduce the space of finite words with a distinguished ending as 
$$\tilde{V}_A:=\{\alpha.\beta\in V_A\times \{1,\ldots,N\}: \alpha\beta\in V_A\}=\{\gamma\in I_A: |s(\gamma)|=1\},$$
and write $\ell^2(\tilde{V}_A)$ for the associated Hilbert space.
Furthermore, consider the Hilbert space 
$$F_A:=\bigoplus_{\gamma\in \tilde{V}_A}\ker \Delta_{\gamma}\subset L^2(G_A,\mu_{G_A}),$$ 
and the associated projection $P_A$ from $L^2(G_A,\mu_{G_A})$ onto $F_A$. We note that for $\gamma\in I_A$,
$$\ker \Delta_{\gamma}=\mathbb C \chi_{G_{\gamma}}.$$ 
Here $\chi_{G_{\gamma}}$ denotes the characteristic function of $G_{\gamma}\subset G_A$. It is therefore useful to recall that up to renormalization we find the first element in the wavelet basis for each $L^2(G_\gamma,\mu_{\gamma}),$ 
$$\mathrm{e}_{\gamma}=\mu (G_{\gamma})^{-1/2}\chi_{G_{\gamma}}.$$ 
By construction, $(\mathrm{e}_{\gamma})_{\gamma\in \tilde{V}_A}$ forms an orthonormal basis for $F_A$. In particular, $P_AP_\gamma=P_\gamma P_A$ is the rank one projection onto $\mathrm{e}_\gamma$:
$$P_AP_{\gamma}f=\langle P_{\gamma}f, \mathrm{e}_\gamma\rangle \mathrm{e}_\gamma=\langle f, \mathrm{e}_\gamma\rangle \mathrm{e}_\gamma.$$
The map $\delta_{\gamma}\mapsto \mathrm{e}_\gamma$ extends to a unitary isomorphism $v:\ell^2(\tilde{V}_A)\to F_{A}$ and $vv^{*}=P_{A}$.

\begin{prop}
\label{lknljnada}
For $a\in\LC$ the operators $(1-P_{A})aP_{A}$ are of finite rank.
\end{prop}

\begin{proof}
We have 
\[S_i e_\gamma=\left\{\begin{matrix}e_{i\gamma} &\textnormal{if   }&i\gamma\in \tilde{V}_A\\
0&\textnormal{if   } &i\gamma\notin \tilde{V}_A,\end{matrix}\right.\]
Now $P_AS_i^{*} e_{j\gamma}=\delta_{ij}e_{\gamma}$ and $S_{i}^{*}P_A-P_AS_{i}^{*}P_A$ is of finite rank. Since $\LC$ is generated by $\{S_{i}:i=1,\cdots, N\}$ as a $*$-algebra, the statement follows.
\end{proof}

Also, denote by $\rho_A$ the left regular representation of $O_A$ on $L^2(G_A,\mu_{G_A})$. 

\begin{prop}
The triple $(L^2(G_A,\mu_A),\rho_A,2P_A-1)$ is an odd Fredholm module over $O_A$.
\end{prop}

\begin{proof}Since $C_{c}^{\infty}(G_{A})$ is a $*$-algebra and 
$$[P_A,a]=-(1-P_{A})aP_{A}+((1-P_{A})a^*P_{A})^*,$$
this follows from Proposition \ref{lknljnada}.
\end{proof}

We note that the Fredholm module $(L^2(G_A,\mu_A),\rho_A,2P_A-1)$ is finitely summable, even with finite rank commutators against $\LC$, by the proof of Proposition \ref{lknljnada}. This stands in sharp contrast to the lack of finitely summable spectral triples on $O_A$ as discussed in \cite{GMCK}. It is simple to check the following using Kaminker-Putnam's $KK$-duality \cite{KPDual}, with the same argument as in \cite{DGMW,GMCK}.

\begin{prop}
The class $[L^2(G_A,\mu_A),\rho_A,2P_A-1]\in K^1(O_A)$ is the image of the unit $[1]\in K_0(O_{A^T})$ under the $KK$-duality isomorphism $K_0(O_{A^T})\to K^1(O_A)$.
\end{prop}

More generally, for a finite word $\beta\in V_A\setminus \{{\o}\}$ we can consider the Hilbert space 
$$F_{A,\beta}=\bigoplus_{\gamma\in I_A: s(\gamma)=\beta}\ker \Delta_{\gamma}\subset L^2(G_A,\mu_{G_A}),$$ 
and the associated projection $P_{A,\beta}$ from $L^2(G_A,\mu_{G_A})$ onto $F_{A,\beta}$. By the same argument as above, we conclude the following.

\begin{prop}
\label{kclass}
For every $\beta\in V_A\setminus \{{\o}\}$, the triple $(L^2(G_A,\mu_A),\rho_A,2P_{A,\beta}-1)$ is an odd Fredholm module over $O_A$. Also, the associated class $[(L^2(G_A,\mu_A),\rho_A,2P_{A,\beta}-1)]\in K^1(O_A)$ is the image of the projection $[S_\beta S_\beta^*]\in K_0(O_{A^T})$ under the $KK$-duality isomorphism $K_0(O_{A^T})\to K^1(O_A)$.
\end{prop}

We note that the classes $[S_\beta S_\beta^*]\in K_0(O_{A^T})$ generate  $K_0(O_{A^T})$. Indeed, we have a Murray--von Neumann equivalence $S_\beta S_\beta^*\sim_{\rm MN} S_\beta^* S_\beta$. A direct computation gives $S_\beta^*S_\beta=S_j^*S_j\sim_{MN}S_jS_j^*$ where $j=\beta_{|\beta|}$ is the last letter of $\beta$, so that
$$[S_\beta S_\beta^*]=[S_jS_j^*]\in K_0(O_{A^T}).$$
Finally \cite[Proposition 3.1]{CK2} shows that the map $\mathbb{Z}^N\to K_0(O_{A^T})$, $(k_j)_{j=1}^N\mapsto \sum_{j=1}^N k_j [S_jS_j^*]$ is surjective (with kernel $(1-A)\mathbb{Z}^N$). 

\subsection{Spectral triple representative}
\label{subsec:spectrip}

We recall the proper continuous function $\ell:G_A\to \mathbb{N}$ from \eqref{eq:scale}. In previous works \cite{GMCK,GMON}, one used the length function $|c|+\kappa$  constructed from the maps $\kappa:G_A\to \mathbb N$ and $c:G_A\to \mathbb Z$. However, for $g\in G_\gamma$ we have that
$$|c(g)|+\kappa(g)=
\begin{cases}
\ell(g)-|s(\gamma)|, \; &|r(\gamma)|+1>|s(\gamma)|,\\
2\ell(g)-2-3|r(\gamma)|\; &|r(\gamma)|+1\leq |s(\gamma)|
\end{cases},$$
so $\ell$ and $|c|+\kappa$ are mutually relatively bounded by each other. 

Consider now the unbounded operator $D$ on $\LC$ given by
\begin{equation}\label{eq:D}
Df=
\begin{cases}
-(\Delta+M_{\ell})f, &\text{if } f\in F_A^{\perp}\\
M_{\ell}f, &\text{if } f\in F_A
\end{cases}.
\end{equation}
We see from Theorem \ref{kljnkjnadadjbnojbn} that in the orthonormal basis $(\mathrm{e}_{\gamma},\mathrm{e}_{(\gamma,\nu,j)})_{(\gamma,\nu,j)\in \mathfrak{I}_A}\subset \LC$  for $L^2(G_A)$ from Theorem \ref{lknlknadasd}, we have 
$$D\mathrm{e}_{\gamma}=
\begin{cases}
-|\gamma|\mathrm{e}_{\gamma}, &\text{if } |s(\gamma)|>1,\\
|\gamma|\mathrm{e}_{\gamma}, &\text{if } |s(\gamma)|=1
\end{cases}.
$$
and 
$$D\mathrm{e}_{(\gamma,\nu,j)}=-(|\gamma|+\lambda_{s(\gamma)}^A(\nu))\mathrm{e}_{(\gamma,\nu,j)},$$
for $\lambda_{s(\gamma)}^A(\nu)$ as in Lemma \ref{complemlog}. Note that 
$$|\gamma|+\lambda_{s(\gamma)}^A(s(\gamma)\nu_0)\simeq |\gamma|+|\nu_0|,$$ 
by Proposition \ref{asjnkjnada}. We think of this spectral property as the splitting of 
$$|D|=\bigoplus_{\gamma\in I_A} (|\gamma|+\Delta_\gamma),\quad\mbox{under}\quad L^2(G_A,\mu_{G_A})=\bigoplus_{\gamma\in I_A} L^2(G_\gamma,\mu_{\gamma}),$$
 and where each $\Delta_\gamma\geq 0$ has compact resolvent with eigenvalues indexed by a one-point set and $(\nu_0,j)$ from a subset of $V_A\times\{1,\ldots,N\}$ where the size of each eigenvalue is $|\nu_0|$. In particular, $D$ has compact resolvent and $\mathrm{Tr}(\mathrm{e}^{-t|D|})<\infty$ for $t>\log(\lambda_{\max})/C_0$, where $C_0>0$ is the constant appearing in Proposition \ref{asjnkjnada}. \color{black} Since $P_A$ commutes with $M_\ell$ and $\Delta$, we can simplify our notation and write 
$$D=(2P_A-1)(\Delta+M_{\ell}),$$
as an unbounded self-adjoint operator. 

\begin{prop}
\label{lknlknlknad2}
The triple $(\LC, L^2(G_A,\mu_{G_A}), D)$ is an odd spectral triple such that 
$$D=(2P_A-1)(\Delta+M_{\ell}).$$
 Consequently, the $K^1$-homology class of it equals $[L^2(G_A,\mu_{G_A}),\rho_A,2P_A-1]\in K^1(O_A)$. And more generally, any finite word $\beta\in V_A$ defines an odd spectral triple 
 $$(\LC, L^2(G_A,\mu_{G_A}), (2P_{A,\beta}-1)(\Delta+M_{\ell})),$$ 
 whose $K^1$-homology class equals $[L^2(G_A,\mu_{G_A}),\rho_A,2P_{A,\beta}-1]\in K^1(O_A).$
\end{prop}

\begin{proof}
We shall prove the statement for $P_A$ and the proof about $P_{A,\beta}$ is similar. First, it is clear from the definition of $D$ in \eqref{eq:D} that $D=(2P_A-1)(\Delta+M_{\ell})$. Also, from the preceding discussion we see that $D$ has compact resolvent. We only need to show that $[D,S_i]$ extends to a bounded operator for every $i\in \{1,\ldots, N\}$. To this end, observe that $$D=(2P_A-1)M_{\ell}-\Delta.$$ From Proposition \ref{prop:bounded_com_Lap} we have that each $[\Delta,S_i]$ extends to a bounded operator. Furthermore, each $[(2P_A-1)M_{\ell},S_i]$ extends to a bounded operator as well, which can be seen as follows. We again view each $[(2P_A-1)M_{\ell},S_i]$ as a matrix operator using Proposition \ref{descmatrdecomp}. The matrix decomposition of the operator $(2P_A-1)M_{\ell}$ is diagonal and we derive that there is a bounded function $\zeta_{i}:I_{A}\times I_{A}\to \mathbb{R}$ such that $[(2P_A-1)M_{\ell},S_{i,\gamma}^{\gamma'}]=\zeta_{i}(\gamma, \gamma')S_{i,\gamma}^{\gamma'}$. We conclude that $[(2P_A-1)M_{\ell},S_i]$ is bounded.   
\end{proof}

\section{The isometry group and entropy}
\label{secljnjasn}

We now describe the isometry group of the spectral triple $(\LC, L^2(G_A,\mu_{G_A}), D)$ from Proposition \ref{lknlknlknad2}. We write $\mathsf{G}_D$ for the isometry group of $(\LC, L^2(G_A,\mu_{G_A}), D)$ as defined in \cite{park}. The group $\mathsf{G}_D$ is the subgroup of the automorphism group of $O_A$ consisting of automorphisms $\varphi $ of $O_A$ for which there is a unitary implementer $U$ of $\varphi $ on $L^2(G_A,\mu_{G_A})$ such that $UDU^*=D$ in the sense of unbounded operators, for more details see \cite{park}. The operator $U$ is uniquely determined by $\varphi $ up to an element of the unitary group $U(O_A')$ on the dense subspace $O_A\subseteq L^2(G_A,\mu_{G_A})$ (cf. Remark \ref{cyclicremark}) from 
\begin{equation*}\label{eq:implementer}
U(\xi)=\varphi (\xi), \quad \xi\in O_A.
\end{equation*}
In fact, an automorphism $\varphi $ of $O_A$ extends to a unitary on $L^2(G_A,\mu_{G_A})$ if and only if the KMS-state on $O_A$ for the gauge action is $\varphi $-invariant; which in turn is equivalent to $\varphi $ commuting with the gauge action (i.e. graded). 

Before diving into deeper results about  the isometry group, we start with some easy observations allowing us to embed $\mathsf{G}_D$ into the unitary group $U(N)$. Our argument closely follows \cite{contirossi} describing a similar problem on the algebra $O_N$. We note the following: 
\begin{enumerate}
\item $1$ is an eigenvalue of $D$ with eigenspace spanned by $\mathrm{e}_{{\o}.\beta}$ for $|\beta|=1$. We have that such an $\mathrm{e}_{{\o}.\beta}$ is a normalization of the characteristic function for the set 
$$
G_{{\o}.\beta}= \{(x,0,x)\in G_A:  \; x\in C(\beta)\}.
$$
In terms of the generators of $O_A$, $\mathrm{e}_{{\o}.\beta}=S_\beta S_\beta^*$. 
\item $2$ is an eigenvalue of $D$ with eigenspace spanned by $\mathrm{e}_{\alpha.\beta}$ for $|\alpha|=1$ and $|\beta|=1$ with $\alpha \beta \in V_A$. We have that such an $\mathrm{e}_{\alpha.\beta}$ is a normalization of the characteristic function for the set 
$$
G_{\alpha.\beta}= \{(x,1,y)\in G_A: x=\alpha y, \; y\in C(\beta)\}.
$$
In terms of the generators of $O_A$, $\mathrm{e}_{\alpha.\beta}=S_\alpha S_\beta S_\beta^*$.
\end{enumerate}
In particular, if $\varphi \in \mathsf{G}_D$, the eigenspace corresponding to the eigenvalue $1$ must be preserved. So there is a unitary matrix $u^{(1)}=(u^{(1)}_{\beta,\beta'})_{\beta,\beta'=1}^N\in U(N)$ such that 
$$\varphi (S_\beta S_\beta^*)=\sum_{\beta'=1}^N u^{(1)}_{\beta',\beta} S_{\beta'} S_{\beta'}^*.$$
In fact, given that $\varphi$ is an automorphism, $u^{(1)}$ must be a permutation matrix (i.e. only one non-zero entry per row and column which is $1$). Similarly, there must be a unitary matrix $u^{(2)}=(u^{(2)}_{\alpha,\alpha'})_{\alpha,\alpha'=1}^N\in U(N)$ such that for every $\alpha,\beta \in \{1,\ldots, N\}$ it holds
\begin{align*}
\varphi (S_\alpha S_\beta S_\beta^*)&=\sum_{\alpha',\beta'=1}^Nu^{(2)}_{\alpha',\alpha}u^{(1)}_{\beta',\beta}S_{\alpha'}S_{\beta'} S_{\beta'}^*\\
&=\sum_{\alpha'=1}^N u^{(2)}_{\alpha',\alpha} S_{\alpha'}\varphi (S_{\beta}S_{\beta}^*).
\end{align*}
Therefore, summing over $\beta$ we obtain 
$$\varphi (S_\alpha)=\sum_{\alpha'=1}^N u^{(2)}_{\alpha',\alpha} S_{\alpha'}.$$
Since $\{S_\alpha\}_{\alpha=1}^N$ generate $O_A$, the map\begin{equation}\label{eq:embedding}
\iota:\mathsf{G}_D\hookrightarrow U(N),\qquad \varphi\mapsto u^{(2)},
\end{equation}
is a group embedding. We aim to describe this embedding explicitly. To this end, we first obtain an algebraic description of the closed subgroup $\mathsf{G}_A$ of $U(N)$ consisting of all unitaries $u\in U(N)$ such that the map $\varphi_u(S_\alpha):=\sum_{\alpha'=1}^N u_{\alpha',\alpha} S_{\alpha'}$ for $\alpha \in \{1,\ldots, N\}$ extends to an automorphism of $O_A$.

\begin{prop}\label{prop:G_A}
The group $\mathsf{G}_A$ consists of all unitaries $u\in U(N)$ such that 
$$(u\otimes \widetilde{S}u) \widetilde{A} (u^*\otimes u^*\widetilde{S}^*)= (1\otimes \widetilde{S}) \widetilde{A} (1\otimes \widetilde{S}^*),$$ 
where $\widetilde{A}\in M_N(\mathbb C)\otimes M_N(\mathbb C)$ is given on the standard basis $\{e_i\otimes e_j\}_{i,j=1}^{N}$ of $\mathbb C^N\otimes \mathbb C^N$ by 
$$\widetilde{A}(e_i\otimes e_j)=A_{i,j}e_i\otimes e_j,$$
 and $\widetilde{S}\in M_N(O_A)$ is the diagonal matrix $\widetilde{S}e_i=S_ie_i.$
\end{prop}

\begin{proof}
We describe the elements $u\in U(N)$ for which the map $\varphi_u(S_\alpha):=\sum_{\alpha'=1}^N u_{\alpha',\alpha} S_{\alpha'}$ extends to an automorphism of $O_A$; that is, $\{\varphi_u(S_\alpha)\}_{\alpha=1}^N$ satisfies the Cuntz--Krieger relations. Since $u$ is unitary one has
$$\sum_\alpha \varphi_u(S_\alpha)\varphi_u(S_\alpha)^*=1.$$
Moreover, it holds that 
\begin{align*}
\varphi_u(S_\alpha)^*\varphi_u(S_\beta)=&\sum_{\alpha',\beta'} \overline{u_{\alpha',\alpha}}u_{\beta',\beta} S_{\alpha'}^* S_{\beta'}=\sum_{\alpha',\beta'} \overline{u_{\alpha',\alpha}}u_{\alpha',\beta} A_{\alpha',\beta'} S_{\beta'} S_{\beta'}^*.
\end{align*}
Therefore, when $\alpha \neq \beta$, we have that $\varphi_u(S_\alpha)^*\varphi_u(S_\beta)=0$ if and only if for every $\beta'\in \{1,\ldots N\}$ one has 
$$\sum_{\alpha'}\overline{u_{\alpha',\alpha}}u_{\alpha',\beta}A_{\alpha',\beta'}=0.$$ 
On the other hand, 
$$\sum_{\alpha'}A_{\alpha,\alpha'}\varphi_u(S_{\alpha'})\varphi_u(S_{\alpha'})^*=\sum_{\alpha',\beta',\beta''}\overline{u_{\beta'',\alpha'}}u_{\beta',\alpha'} A_{\alpha,\alpha'} S_{\beta'}S_{\beta''}^*.$$ 
Consequently, we have that $$\varphi_u(S_\alpha)^*\varphi_u(S_\alpha)=\sum_{\alpha'}A_{\alpha,\alpha'}\varphi_u(S_{\alpha'})\varphi_u(S_{\alpha'})^*$$ 
if and only if when $\beta'\neq \beta''$ it holds that
$$\sum_{\alpha'}\overline{u_{\beta'',\alpha'}}u_{\beta',\alpha'}A_{\alpha,\alpha'}S_{\beta'}S_{\beta''}^*=0$$
and for $\beta'=\beta''$ it holds 
$$\sum_{\alpha'}\overline{u_{\beta',\alpha'}}u_{\beta',\alpha'}A_{\alpha,\alpha'}=\sum_{\alpha'}\overline{u_{\alpha',\alpha}}u_{\alpha',\alpha}A_{\alpha',\beta'}.$$

In summary, for $u\in U(N)$ the map $\varphi_u$ extends to an automorphism of $O_A$ if and only if 
\begin{equation}\label{eq:isom}
\delta_{\alpha,\beta}\sum_{\alpha'}\overline{u_{\beta'',\alpha'}}u_{\beta',\alpha'}A_{\alpha,\alpha'}S_{\beta'}S_{\beta''}^*=\delta_{\beta',\beta''}\sum_{\alpha'}\overline{u_{\alpha',\alpha}}u_{\alpha',\beta}A_{\alpha',\beta'} S_{\beta'}S_{\beta''}^*.
\end{equation}
Then, a simple calculation shows that \eqref{eq:isom} is equivalent to 
$$(u\otimes \widetilde{S}u) \widetilde{A} (u^*\otimes u^*\widetilde{S}^*)= (1\otimes \widetilde{S}) \widetilde{A} (1\otimes \widetilde{S}^*),$$ 
where $\widetilde{A}$ and $\widetilde{S}$ are the matrices described in the statement. This completes the proof.
\end{proof}

\begin{remark}
For the Cuntz algebra $O_N$ we have $A_{i,j}=1$ for all $i,j$ and hence in this case $\mathsf{G}_A= U(N)$. In general, the group $\mathsf{G}_A$ is a compact Lie group and it is straightforward to check that the associated Lie algebra consists of all skew-hermitian matrices $X\in M_N(\mathbb C)$ such that 
$$ [1_N\otimes \widetilde{S}X\widetilde{S}^*+X\otimes \widetilde{S}\widetilde{S}^*,\widetilde{A}]=0.$$
\end{remark}
We write $\mathrm{P}_N$ for the group of $N\times N$-permutation matrices, which of course is isomorphic to the symmetric group on $N$ elements. Consider now the subgroup $\Aut(A)\subset \mathrm{P}_N$ defined as $$\Aut(A):=\{q\in  \mathrm{P}_N: qA=Aq\}.$$ In other words, $\Aut(A)$ is the automorphism group of the directed graph associated to $A$. For $q\in \Aut(A)$ and $i\in \{1,\ldots,N\}$ we shall denote by $q(i)\in \{1,\ldots,N\}$ the permutation of $i$ under $q$. An immediate observation is that $\Aut(A)$ acts on the space of admissible finite words $V_A$ by fixing the empty word ${\o}$ and acting on $\alpha_1\ldots \alpha_n \in V_A$ by 
\begin{equation}\label{eq:Aut_action}
q\cdot \alpha_1\ldots \alpha_n := q(\alpha_1)\ldots q(\alpha_n),\qquad q\in \Aut(A).
\end{equation}
This action extends to an action on the index set $I_A$ in \eqref{eq:indexset}. For brevity, we shall denote the action of $q\in \Aut(A)$ on $x\in V_A$ or $I_A$ by $q\cdot x$.

Moreover, $\Aut(A)$ acts by conjugation on $\mathbb T^N$, where the latter is viewed as the group of $N\times N$-diagonal matrices with entries in $\mathbb T$. Further, the associated semi-direct product $\mathbb T^N\rtimes \Aut(A)$ embeds as a group into $U(N)$ by sending $(c,q)\mapsto cq$. At this point we identify $\mathbb T^N\rtimes \Aut(A)$ with its image in $U(N)$. Moreover, from Proposition \ref{prop:G_A} we obtain that $\mathbb T^N\rtimes \Aut(A)$ is in fact a subgroup of $\mathsf{G}_A$ because for every $i,j\in \{1,\ldots,N\}$ and $q\in \Aut(A)$ it holds that $A_{i,j}=A_{q(i),q(j)}.$

The action of the group $\mathbb T^N\rtimes \Aut(A)$ can be described dynamically as well. An element $q\in \Aut(A)$ induces a shift equivariant measure preserving isometry
\[v_{q}:\Sigma_{A}\to\Sigma_{A},\quad (x_{n})_{n\in\mathbb{N}}\mapsto (q(x_{n}))_{n\in\mathbb{N}}.\]
This map can be lifted to a measure preserving homeomorphism $V_{q}:G_{A}\to G_{A}$ by setting $V_{q}(x,n,y):=(v_{q}(x),n,v_{q}(y))$ with the property that $V_{q}:(G_{\gamma},\mathrm{d}_\gamma)\to (G_{q\cdot \gamma}, \mathrm{d}_{q\cdot \gamma})$ with $\mu_{\gamma}=\mu_{q\cdot \gamma}$. An element $c\in\mathbb{T}^{N}$ acts on $L^{2}(G_{A},\mu_{G_{A}})$ via
\[cf(g):=c_{\alpha_{1}(g)}\cdots c_{\alpha_{|\alpha|}(g)}c_{\beta_{|\beta|}}\overline{c_{\beta_{1}(g)}\cdots c_{\beta_{|\beta|-1}(g)}}f(g).\]
Here $\alpha(g)$ and $\beta(g)$ are defined in terms of the discretization map $$\ell:G_{A}\to I_{A},\quad \ell(g)=(\ell_{r}(g),\ell_{s}(g)),$$
of Definition \ref{def: discr} by setting $\alpha(g):=\ell_{r}(g)$ and $\beta(g):=\ell_{s}(g)$.
\begin{thm}
\label{thm:Isom_group}
The isometry group $\mathsf{G}_D$ of the spectral triple $(\LC, L^2(G_A,\mu_{G_A}), D)$ is a compact Lie group. In particular, the image of $\mathsf{G}_D$ under the embedding $\iota:\mathsf{G}_D\hookrightarrow U(N)$ defined in \eqref{eq:embedding} coincides with the subgroup $\mathbb T^N\rtimes \Aut(A)$ of $\mathsf{G}_A$.
\end{thm}

\begin{proof}
The group $\mathsf{G}_D$ is a compact Lie group since $\iota(\mathsf{G}_D)$ is a closed subgroup of $U(N)$. Moreover, from Proposition \ref{prop:G_A} we obtain that 
\begin{equation}\label{eq:Isom_group_1}
\iota(\mathsf{G}_D)\subset \mathsf{G}_A,
\end{equation}
since by definition $\mathsf{G}_D$ consists of automorphisms of $O_A$. Further, the $1$-eigenspace of $D$, which is spanned by $S_{\beta}S_{\beta}^*$ for $|\beta|=1$, is preserved by $\mathsf{G}_D$ and hence a simple calculation for $\varphi(S_{\beta}S_{\beta}^*)$, where $\varphi \in \mathsf{G}_D$ is arbitrary, yields that
\begin{equation}\label{eq:Isom_group_2}
\iota(\mathsf{G}_D)\subset \mathbb T^N\rtimes \mathrm{P}_N.
\end{equation}
In particular, from \eqref{eq:Isom_group_1} and \eqref{eq:Isom_group_2} we have that 
\begin{equation}\label{eq:Isom_group_3}
\iota(\mathsf{G}_D)\subset \mathbb T^N\rtimes \Aut(A).
\end{equation}

Now let $u\in \mathbb T^N\rtimes \Aut(A)$. Since in particular $u\in \mathsf{G}_A$, we have that the KMS-state for the gauge action on $O_A$ is $\varphi_u$-invariant, where $\varphi_u:O_A\to O_A$ is the automorphism given on $S_{\alpha}$ for $\alpha\in \{1,\ldots, N\}$ by
\begin{equation}\label{eq:Isom_group_4}
\varphi_u(S_\alpha):=\sum_{\alpha'=1}^N u_{\alpha',\alpha} S_{\alpha'}.
\end{equation}
As a result, by viewing $O_A\subseteq L^2(G_A,\mu_{G_A})$ as a dense subspace (cf. Remark \ref{cyclicremark}), the operator $U$ on $L^2(G_A,\mu_{G_A})$ defined as $$U_u(\xi):=\varphi_u (\xi), \quad \xi\in O_A$$ is a unitary implementing $\varphi_u$. To prove equality for \eqref{eq:Isom_group_3} we have to show that $[U_u,D]=0$. 

To conclude $[U_u,D]=0$ it suffices to show that $U_u$ commutes with the projection $P_A$ on the Fock space $F_A$, the potential $M_{\ell}$ and the logarithmic Dirichlet Laplacian $\Delta$, see Section \ref{sec:NCG of G_A} for the description of these operators. By writing $u=cq$, where $c\in \mathbb T^N$ and $q\in \Aut(A)$, we have that 
$$\varphi_u(S_{\alpha})=u_{q(\alpha),\alpha}S_{q(\alpha)}.$$ 
Therefore, the Fock space $F_A$ and its complement are preserved. In other words, $U_u$ commutes with $P_A$. Similarly, the eigenspaces of $M_{\ell}$ are preserved and hence $U_u$ commutes with $M_{\ell}$. More precisely, we have that $U_u$ preserves the direct sum decomposition $L^2(G_A,\mu_{G_A})=\bigoplus_{\gamma\in I_A} L^2(G_{\gamma},\mu_{\gamma})$ in the sense that $U_u=\oplus_\gamma U_{u,\gamma}$ where for every $\gamma \in I_A$ the unitary 
$$U_{u,\gamma}:L^2(G_{\gamma},\mu_{\gamma})\to L^2(G_{q\cdot \gamma},\mu_{q\cdot \gamma}),$$ 
acts as the composition of unitary induced by the measure preserving isometry $v_{q}$ and a complex scalar. Since each individual logarithmic Dirichlet Laplacian $\Delta_{\gamma}$ commutes with unitaries induced by measure-preserving isometries (see \cite[Lemma 5.9]{GerMes}) we obtain that $U_u$ commutes with $\Delta$ as well. The proof is now complete.
\end{proof}

\begin{remark}
With a similar proof as that of Theorem \ref{thm:Isom_group}, one can prove that for a finite word $\beta\in V_A$ the isometry group of the spectral triple $(\LC, L^2(G_A,\mu_{G_A}), (2P_{A,\beta}-1)(\Delta+M_{\ell}))$ (see Proposition \ref{lknlknlknad2}) is the compact Lie group $\mathbb T^N\rtimes \Aut_\beta(A)$ where $\Aut_\beta(A)\subset \Aut(A)$ is the subgroup of graph automorphisms that preserve $\beta$ viewed as a finite path in the graph. 

\end{remark}

\subsection{Voiculescu's topological entropy of isometries}

Let us study an interesting fact about the group $\mathsf{G}_{A}$ in relation to Voiculescu's topological entropy. The notion of topological entropy dates back to Kolmogorov--Sinai, who defined the entropy of a measure with respect to a self-map of the space. Among other things, entropy is used to distinguish \enquote{probabilistic} (positive entropy) from \enquote{deterministic} (zero entropy) dynamics. In several examples, there is a measure maximizing the entropy and its value coincides with the topological entropy. The latter notion was extended to unital completely positive maps by Voiculescu \cite{voicent}. Let us recall its definition. 

Let $B$ be a unital $C^*$-algebra and write ${\rm FS}(B)$ for the set of all finite subsets of $B$. If $\Omega\in {\rm FS}(B)$ and $\varepsilon>0$, an $\varepsilon$-approximating triple $(\varphi,\psi,C)$ of $\Omega$ consists of a finite-dimensional $C^*$-algebra $C$, and two ucp-maps $\varphi:B\to C$ and $\psi:C\to B$ such that $\|\psi(\varphi(b))-b\|_B<\varepsilon$ for all $b\in \Omega$. We write ${\rm CPA}(B,\Omega,\varepsilon)$ for the set of $\varepsilon$-approximating triples of $\Omega$. If we assume that $B$ is nuclear, e.g. a Cuntz--Krieger algebra, then ${\rm CPA}(B,\Omega,\varepsilon)$ is non-empty for any $\varepsilon>0$ and $\Omega\in {\rm FS}(B)$. We set 
$${\rm rcp}(\Omega,\varepsilon):=\min\left\{{\rm rank}(C): \,(\varphi,\psi,C)\in {\rm CPA}(B,\Omega,\varepsilon)\right\}.$$
Here ${\rm rank}(C)$ is defined as the dimension of a maximal abelian subalgebra in $C$. 

Now let $\alpha:B\to B$ a unital completely positive map (ucp). For $\Omega\in {\rm FS}(B)$ and $n\in \mathbb{N}$, we use the notation $\Omega^{(n)}:=\cup_{k=0}^n \alpha^k(\Omega)$. Voiculescu's topological entropy of $\alpha$ is defined as 
\begin{equation}
\label{vtopent}
{\rm ht}(\alpha):=\sup_{\varepsilon>0, \Omega\in {\rm FS}(B)} \limsup_{n\to \infty} \frac{\log\left({\rm rcp}(\Omega^{(n)},\varepsilon)\right)}{n}.
\end{equation}

\begin{thm}
\label{vanishihsidnadin}
For all $u\in \mathsf{G}_A$ the Voiculescu noncommutative topological entropy vanishes on the automorphism $\varphi_u:O_A\to O_A$. 
\end{thm}

\begin{proof}
One directly concludes the theorem for the Cuntz algebra from \cite[Theorem 2.2]{skalzach}. The general case follows from the same methods as in the proof of \cite[Theorem 2.2]{skalzach} using that $\mathsf{G}_A$ acts via elements of the matrix algebra in the core spanned by $\{S_iS_j^*: i,j=1,\ldots,N\}$.
\end{proof}

The construction of topological entropy of self-maps of a compact metric space can be found in for instance \cite[Chapter 3.1.b]{Katok_Has}, with the relation to the construction \eqref{vtopent} made clearer in \cite[Exercise 3.1.8 and 3.1.9]{Katok_Has}. By \cite[Proposition 4.8]{voicent}, if $B=C(X)$ for a compact metric space $X$ and $\alpha(f):=g^*(f)$ for a continuous map $g:X\to X$, then ${\rm ht}(\alpha)$ coincides with the topological entropy of $g$. In particular, if $g$ is isometric then ${\rm ht}(\alpha)=0$; this is clear from the construction in \cite[Chapter 3.1.b]{Katok_Has}. This well-known fact about isometries on compact metric spaces turns out to extend to isometries of Cuntz--Krieger algebras.

\begin{cor}
\label{cor: vanishihsidnadin}
The Voiculescu noncommutative topological entropy vanishes on $\mathsf{G}_D$. 
\end{cor}

\begin{proof}
This now follows from Theorems \ref{vanishihsidnadin} and \ref{thm:Isom_group}.
\end{proof}

\end{document}